\documentclass[10pt, a4paper]{article}

\usepackage{amssymb, amsmath, amsthm,mathrsfs}
\usepackage{mathbbol, graphicx, latexsym, tkz-graph}

\usepackage[margin = 3 cm]{geometry}

\newtheorem{theorem}{Theorem}[section]
\newtheorem{lemma}[theorem]{Lemma}
\newtheorem{corollary}[theorem]{Corollary}
\newtheorem{definition}{Definition}[section]
\theoremstyle{remark}
\newtheorem{remark}[theorem]{Remark}
\newtheorem{example}[theorem]{Example}

\begin{document}
\title{A Poncelet Criterion for special pairs of conics in $PG(2,p)$}

\author{Norbert Hungerb\"uhler (ETH Z\"urich)\\ Katharina Kusejko\footnote{E-mail: katharina.kusejko@math.ethz.ch} (ETH Z\"urich)}

\date{10. September 2014}
\maketitle

\begin{abstract}\noindent 
  We study  Poncelet's Theorem in finite  projective coordinate planes
  over  the field  $GF(p)$ and  concentrate on  a particular pencil of
  conics. For pairs of such conics  we investigate whether we can find
  polygons  with $n$  sides,  which  are inscribed  in  one conic  and
  circumscribed  about the  other, so-called  Poncelet Polygons.   By
  using suitable elements of the dihedral group for  these pairs, we prove that the
  length $n$ of such Poncelet  Polygons is independent of the starting
  point. In this sense Poncelet's  Porism is valid. By  using Euler's
  divisor  sum  formula  for  the  totient function,  we  can  make  a
  statement about  the number  of different  conic pairs,  which carry
  Poncelet  Polygons  of  length  $n$.  Moreover,  we  will  introduce
  polynomials  whose  zeros in  $GF(p)$  yield  information about  the
  relation of a given pair of conics. In particular, we can decide for
  a given integer  $n$, whether and how we can find Poncelet Polygons
  for pairs of conics in the given coordinate plane.  We will see that
  this condition  is closely  connected with  the theory  of quadratic
  residues.
\end{abstract}

\section*{Introduction}
In 1813 Jean-Victor Poncelet stated  one of the most beautiful results
in projective geometry, known as Poncelet's Porism. He proved that for
two conics  $C$ and $D$  in the  real projective plane,  the condition
whether  a polygon  with  $n$ sides,  which is  inscribed  in $D$  and
circumscribed about $C$,  is independent of the starting  point of the
polygon. In particular, there cannot be such polygons with a different
number of  sides for a  given pair of  conics. A remarkable  number of
different proofs can  be found in the literature,  ranging from rather
elementary proofs for  special cases up to proofs using  the theory of
elliptic curves or measure theory (see~\cite{DR2011} for an overview).
In addition to  proving the statement itself, much work  has been done
to find  criteria for  the existence  of such  polygons for  two given
conics, the most advanced result given  by Arthur Cayley in 1853.  The
aim of  this paper  is to  look at Poncelet's  Theorem for  a specific
pencil of conics in finite  projective coordinate planes $PG(2,p)$, $p$
an odd prime.   We describe a criterion for the  existence of Poncelet
Polygons in such  planes.  The most interesting part  in this analysis
is the  connection between the  existence of Poncelet  Polygons, which
can  be  seen  as  geometric  objects, and  the  theory  of  quadratic
residues, which is purely number theoretic.

In the first chapter, we introduce  some basic notation as well as the
most important definitions  and results used later  on. In particular,
we  recall some  facts about  conics in  finite projective  planes and
collineations. In the second chapter, we describe the pencil of conics
we are working with  as well as its properties. We  state a version of
Poncelet's Theorem for $PG(2,p)$ and give  a proof for the conic pairs
constructed before.   Moreover, some crucial properties  which help to
find conditions about Poncelet Polygons  are considered.  In the third
chapter,  we start  with a  condition  for the  existence of  Poncelet
Triangles and  reveal a  first connection to  the theory  of quadratic
residues. Also, we show how to  reduce the problem of finding Poncelet
Polygons to  those with  an odd  number of  sides.  A  generic example
using $5$-sided Poncelet Polygons shows how the Euler totient function
is involved  in Poncelet's Theorem  in $PG(2,p)$.  This enables  us to
deduce  more  information about  the  relation  of Poncelet  Pairs  of
conics,  which finally  leads to  the  algorithm that  allows to  find
Poncelet Polygons by looking at the  conic equations only. In the last
chapter, we take  a brief look at the Euclidean  plane and investigate
some parallels to  the formulas derived for the finite  planes, as for
example the half-angle formula, which can henceforth be interpreted in
finite planes as well.

\section{Preliminaries}
We start  with recollecting the  most important definitions  and facts
about  finite   projective  planes  used   later  on  in   this  paper
(see~\cite{MR1612570}).
\begin{definition}  The   triple  $(\mathbb{P},\mathbb{B},\mathbb{I})$
  with  $\mathbb{I} \subset  \mathbb{P} \times  \mathbb{B}$ is  called
  \emph{projective plane}, if the following axioms are satisfied:
\begin{enumerate}
\item For  any two elements  $P,Q \in  \mathbb{P}$, $P \neq  Q$, there
  exists  a  unique  element  $g   \in  \mathbb{B}$  with  $(P,g)  \in
  \mathbb{I}$ and $(Q,g) \in \mathbb{I}$.
\item For  any two elements  $g,h \in  \mathbb{B}$, $g \neq  h$, there
  exists  a  unique  element  $P   \in  \mathbb{P}$  with  $(P,g)  \in
  \mathbb{I}$ and $(P,h) \in \mathbb{I}$.
\item There are four  elements $P_1,\ldots,P_4\in\mathbb{P}$ such that
  $\forall   g\in\mathbb{B}$  we   have   $(P_i,g)\in\mathbb  I$   and
  $(P_j,g)\in\mathbb I$  with $i\neq j$  implies $(P_k,g)\notin\mathbb
  I$ for $k\neq i,j$.
\end{enumerate}
\end{definition}
We are  only working with finite  projective planes of order  $p$, for
$p$ an odd prime, i.e.:
\begin{definition}        A       finite        projective       plane
  $(\mathbb{P},\mathbb{B},\mathbb{I})$  is said  to be  of \emph{order
    $p$},  if   $\  \left|\mathbb{P}\right|=\left|\mathbb{B}\right|  =
  p^2+p+1$. It will be denoted by $\mathcal{P}_p$.
\end{definition}
A  particular   class  of  finite  projective   planes  are  so-called
coordinate planes, which are constructed as follows:
\begin{enumerate}
\item The set of points $\mathbb{P}$ is defined as
	$$ \mathbb{P} = \left\{(x,y,z) \in GF(p)^3, \ (x,y,z) \neq (0,0,0)\right\} / \sim $$
        where $\sim$ is an equivalence relation given by
$$ (\lambda x,\lambda y,\lambda z) \sim (x,y,z), \ \forall \lambda \in GF(p)^*=GF(p) \backslash \left\lbrace 0\right\rbrace $$
\item  Using  the   same  equivalence  relation,  the   set  of  lines
  $\mathbb{B}$ is defined as
$$ \mathbb{B} = \left\{(a,b,c) \in GF(p)^3, \ (a,b,c) \neq (0,0,0)\right\} / \sim $$
\item  The incidence  relation  $\mathbb{I}$ is  given  by the  inner
  product:
$$P=(x,y,z) \in \mathbb{P} \text{ is incident with } g=(a,b,c) \in \mathbb{B} \ \Longleftrightarrow \ ax+by+cz=0 \text{ in } GF(p) $$
\end{enumerate}
The finite  projective plane of order  $p$ constructed in this  way is
unique up to isomorphisms and denoted by $PG(2,p)$. All points, lines,
pairs of lines and conics in $PG(2,p)$ can be described as solutions of
\begin{equation} \label{QuForm} a x^2 + b y^2 + c z^2 + d xy + e x z+f yz =0,
\end{equation}
with $a,b,c,d,e,f  \in GF(p)$. Another  way to look at  this quadratic
form is to consider $v^TAv$ for $v=(x,y,z)$ and
$$ A = \begin{pmatrix} a & \frac{d}{2} & \frac{e}{2} \\ \frac{d}{2} & b & \frac{f}{2} \\ \frac{e}{2} & \frac{f}{2} & c \end{pmatrix}$$
Equation (\ref{QuForm}) then corresponds to a conic if and only if the
corresponding matrix  $A$ is  regular. If  equation (\ref{QuForm})
corresponds  to a  singular matrix  and is  irreducible, the  solution
is one  point  only.  Otherwise,  if  the quadratic  form
splits into  two linear factors, it  corresponds to one or  two lines.
In  abuse of  notation, by  a conic  we mean  the set  of points,  the
quadratic  equation as  well  as the  corresponding symmetric  matrix,
depending on  the context. By the Chevalley-Warning Theorem, (\ref{QuForm}) has at least
one non-zero solution $P=(x,y,z)$. If the corresponding matrix is regular, 
by cutting the conic with lines passing through $P$, it is easy to
see that there are exactly $p+1$ points on a conic.
As  usual, if $O$ is  a given conic,  $P$ a
point and $l$ a line, we call $l$ a {\em tangent}, if it has one point
in common  with $O$, a  {\em secant}, if it  has two points  in common
with $O$ and  an {\em external line}  if it misses $O$.  $P$ is called
{\em inner point}, if there is no  tangent to $O$ through $P$, and {\em
  exterior point}, if  there are two tangents from $P$  to $O$.  It is
very convenient to work with the  matrix representation of a conic, as
can be seen by the following fact:
\begin{lemma} \label{AP} Let $O$ be any conic in $PG(2,p)$ and $P$ be any point in $PG(2,p)$. Then:
\begin{itemize}
	\item $P$ is on $O$ $\Leftrightarrow$ $O P$ is a tangent of $O$
	\item $P$ is an exterior point of $O$ $\Leftrightarrow$ $O P$ is a secant of $O$
	\item $P$ is an inner point of $O$ $\Leftrightarrow$ $OP$ is an external line of $O$
\end{itemize}
\end{lemma}
An important tool are collinear maps of $PG(2,p)$:
\begin{lemma} If $S$ is a regular $3\times 3$ matrix with coefficients
  in $GF(p)$, then $\phi_S:  PG(2,p) \rightarrow PG(2,p),\ P \mapsto
  SP$ is bijective and collinear, i.e., a set of collinear points is mapped to a set
  of collinear points.
\end{lemma}

\begin{remark}
  We    may   consider    the   map    $\phi_S$   as    a   coordinate
  transformation. Observe, that for three points $P,Q,R$ in $PG(2,p)$,
  we have:
$$ SP, SQ, SR \text{ not collinear }\Leftrightarrow\det(SP,SQ,SR) = \det(S) \det(P,Q,R) \neq 0$$
For a point  $P$ and  a line $l$ in $PG(2,p)$, we have:
$$ P\in l \Leftrightarrow SP \in (S^T)^{-1}l$$
Moreover, let $P_1,...,P_{p+1}$ be the $p+1$ points on a conic $O$ and $S$
regular. Then the points $SP_1,...,SP_{p+1}$ lie again on a conic, namely on the
 conic    $\phi_S(O)$  given    by
${(S^{-1})}^{T}O S^{-1}$. In fact, we have
$$ SP\in \phi_S(O) \Leftrightarrow (SP)^{T} {(S^{-1})}^{T}O S^{-1} (SP)= 0 
\Leftrightarrow  P^{T}OP = 0 \Leftrightarrow P \in O $$
\end{remark}

\section{A special pencil of conics in $PG(2,p)$}
\subsection{Construction and properties}
In all of the following, we only consider conics of the form
\begin{equation} \label{ConicEquation}
O_k: x^2+k y^2+c k z^2=0,\ 1\leq k \leq p-1,
\end{equation}
where the parameter $c$ is a nonsquare if $p\equiv 1(4)$, and  a square if
$p\equiv 3(4)$. To  understand the properties of a pair of such conics, we
first  have  a closer  look  at  a  specific  partition of  the  plane
$PG(2,p)$. The  idea is to  start with  the point $P=(1,0,0)$  and the
line $g$ through the points $(0,1,0)$ and $(0,0,1)$. By looking at all
linear combinations of the equations corresponding to $P$ and
$g$, we  get a partition of  the plane consisting of  the $p-1$ conics
$O_1$, ... ,$O_{p-1}$  as well as the  point $P$ and the  line $g$. To
see this, look at the following results:

\begin{lemma} \label{EquationP}  An equation of the  point $P=(1,0,0)$
  in the plane $PG(2,p)$ is given by
\begin{equation} 
P: y^2+c z^2 =0,\label{point}
\end{equation}
for $p-c$ not a square in $GF(p)$. In particular, for $p \equiv 1(4)$,
all nonsquares and for $p \equiv  3(4)$, all squares in $GF(p)$ can be
used as the parameter $c$.
\end{lemma}
\begin{proof} 
  $P=(1,0,0)$ clearly solves equation~(\ref{point}). As the associated
  matrix is singular, it describes a point, a line or a pair of lines.
  It is  a point, if  the polynomial  $y^2+c z^2$ is  irreducible over
  $GF(p)$. This is the case if and only if $p-c$  is not a square in $GF(p)$.  By
  a well-known theorem in number theory (see \cite{MR2445243}), we have:
$$ p \equiv 1 (4) \Leftrightarrow \exists \ u, 1 \leq u<p: \ u^2 \equiv p-1 (p) $$
Using this result, we obtain:
$$ c \equiv k^2 (p) \Leftrightarrow p-c \equiv p-k^2 (p) \Leftrightarrow p-c \equiv (uk)^2(p), \text{ for } u^2\equiv p-1(p) $$
So $c$ is  a square if and only  if $p-c$ is a square for  $p \equiv 1
(4)$. Because of this, we can  only choose nonsquares for the equation
of $P=(1,0,0)$ in such planes. Since $p-1$ is not a square in $GF(p)$,
$p \equiv 3 (4)$, we have to choose the squares in that case.
\end{proof}
In the construction which follows, we start with any point $P$ and
any line $g$,  $P \notin g$. Since there exists a collineation of $PG(2,p)$, which maps any three noncollinear points to any other noncollinear points, we restrict  the proofs, without
loss of generality, to $P=(1,0,0)$  and $g$ the line through $(0,1,0)$
and $(0,0,1)$.
\begin{lemma} \label{UniqueInnerPointP}  Let $P$ be a point and $g$
  a line  in $PG(2,p)$,  such that  $P \notin  g$. Then  there exist
  $p-1$   conics   $O_1$,...,$O_{p-1}$,   such  that   the   equations
  corresponding to  $P, g, O_1,...,O_{p-1}$ are  closed under addition
  and  the zeros  of these  equations form  a partition  of the  plane
  $PG(2,p)$. Moreover, $P$ is the  unique point in $PG(2,p)$, which is
  an inner point of all conics $O_1$,...,$O_{p-1}$. \end{lemma}
\begin{proof} Without loss of generality, take $P=(1,0,0)$ and $g$ the
  line through $(0,1,0)$ and  $(0,0,1)$. By Lemma \ref{EquationP}, the
  corresponding equations are given by
$$ g: x^2=0 \text{ and } P: y^2+c z^2=0,$$
for $p-c$  not a square. Considering all nontrivial $GF(p)$-linear combinations of $P$ and $g$ leads  to $p-1$
conics, where we define $O_1:= P +  g$ and $O_k := P + O_{k-1}$, i.e.,
$$O_k: x^2+k  y^2+k c z^2 =  0\text{ for }k=1,\ldots,p-1$$  
Note  that these
equations are  all closed under  addition. Moreover, the  solutions of
these  equations are  disjoint, since  a  common solution  of any  two
equations would imply a common solution  of $P$ and $g$ as well, which
contradicts our  assumption. Because of  this, the solutions  of these
$p+1$  equations give  $p(p+1)+1=p^2+p+1$ distinct  points, which  are
indeed  all points  of $PG(2,p)$.  Hence, the  solutions of  the $p+1$
equations form a partition of $PG(2,p)$.

For the second statement,  we first have to show that  $P$ is an inner
point  of all  conics.  Note  that any  point on  $O_k$ has  a nonzero
$x$-coordinate, since  otherwise there  would be an  intersection with
the line $g$.   By Lemma \ref{AP}, for any point  $R = (1,r_2,r_3)$ on
$O_k$ the  tangent to $O_k$  in $R$ is the  line $(1,k r_2,k  c r_3)$.
The point $P$ is  an inner point of $O_k$, if it  is not incident with
any tangent of $O_k$. Indeed, we have
$$(1,0,0)\cdot(1,k r_2,k c r_3)=1 \neq 0$$
and hence,  $P$  is not  incident  with  any tangent  of  $O_k$  for $k  =
1,...,p-1$.  Therefore,  it is  an inner  point  of all  conics
$O_1,...,O_{p-1}$. It remains to show that no other point of $PG(2,p)$
is an inner point of all these conics. Since we can exclude all points
lying  on  some conic  $O_k$,  we  just have  to  deal  with the  line
$g$. Note that:
$$ g = \left\{(0,0,1),(0,1,0),(0,1,1),...,(0,1,p-1)\right\}$$
Consider the point $(0,0,1)$. This  point is incident with the tangent
of $O_k$ in $R  = (1,r_2,r_3)$, if $ (1,k r_2, k c  r_3)\cdot(0,0,1) = k c
r_3 $ is zero. Since $k\neq 0$ and $c\neq 0$, this is exactly the case
for all points of the form $R = (1,r_2,0)$. Because of this, $(0,0,1)$
is not an inner point of those conics containing such points. Since we
have  a partition  of  the  plane, the  existence  of  such conics  is
guaranteed.  We  proceed  similarly  for  the  points  $(0,1,\alpha)$,
$\alpha =  0,...,p-1$. Such a  point is  incident with the  tangent of
$O_k$ in $R  = (1,r_2,r_3)$ if $  (1,k r_2, k c  r_3)\cdot(0,1,\alpha) = k
r_2+k c \alpha r_3 $ is zero. Since $k \neq 0$, we need $r_2+ c \alpha
r_3$ to be  zero. The points $R_\alpha  := (1,\alpha,-c^{-1})$ satisfy
this condition and for all values  of $\alpha$, $R_\alpha$ is indeed a
point on some conic $O_k$, since we have a partition of the plane.
\end{proof}
\begin{theorem} \label{QuSecants}  Let $P$  be a point and  $g$ a
  line  in $PG(2,p)$ with  $P \notin  g$. Let  $O_1$,...,$O_{p-1}$ be  the
  $p-1$  pairwise disjoint  conics  obtained by  adding the  equations of $P$ and $g$
  successively. Then  each line through $P$  is a secant of  all $O_i$,
  $i$ a square in $GF(p)$ and an external line of all $O_j$, $j$ not a
  square in $GF(p)$ or vice versa.
\end{theorem}
\begin{proof} Without loss of generality, take $P=(1,0,0)$ and $g$ the
  line through  $(0,1,0)$ and $(0,0,1)$.  All $p+1$ lines  through $P$
  are given by
$$ s_t: y+t z = 0,\ t=0,...,p-1 \text{ and } \ s_p: z=0. $$
Recall, that $O_k: x^2+k y^2+k c z^2 = 0$ for $k=1,\ldots,p-1$. We
start   proving   the  theorem   in   planes   $PG(2,p)$,  $p   \equiv
1(4)$. Remember that in  these planes, $p-1$ is a square  and $c$ is a
nonsquare. We  start looking  at the  line $s_p$.  This line  is a
secant  of $O_k$  if  the following  equation has  a  solution $y$  in
$GF(p)$:
$$ 1 + k y^2 = 0 \Leftrightarrow y^2 = k^{-1} (p-1) $$
Since $(p-1)$ is a square, the  equation above is solvable if and only
if $k^{-1}$  is a square,  which is  if and only  if $k$ is  a square,
$k\in  \left\{1,...,p-1\right\}$.  Hence, $z=0$  is  a  secant of  all
$O_k$,  $k$ a  square and  an external  line of  the remaining  conics
$O_k$, $k$  a nonsquare.  For the lines  $s_t$, $t=0,...,p-1$,  we see
that $s_t$ is a secant of $O_k$, if the following equation is solvable
for $z$ in $GF(p)$:
$$ 1 + k (t^2+c) z^2 = 0 \Leftrightarrow z^2 = k^{-1} (t^2+c)^{-1} (p-1) $$
Note  that   $t^2+c  \neq  0$,  since   $p-c$  was  chosen  to   be  a
nonsquare. The equation  is solvable, if $k (t^2+c)$ is  a square. For
this, we find:
$$ s_t \text{ is a secant of $O_k$ } \Leftrightarrow  k = \begin{cases}\text{square}, &\text{if $(t^2+c)$ square}\\\text{nonsquare},&\text{if $(t^2+c)$ nonsquare}\end{cases} $$
Hence, $s_t$  is either  a secant  of all  $O_k$, $k$  a square,  or all
$O_k$, $k$  not a  square, depending  on the parameters  $t$ and  $c$. A
similar discussion  covers the  planes $PG(2,p)$ for  $ p  \equiv 3
(4)$.
\end{proof}
The following property is the main building block for what follows:
\begin{lemma}     \label{disjointtangents}None    of     the    conics
  $O_1,...,O_{p-1}$ have tangents in common.
\end{lemma}
\begin{proof} Consider the conics $O_N$ and $O_S$ given by
$$ O_N: x^2+N y^2+c N z^2 = 0 \text{ and } O_S: x^2+S y^2+c S z^2 = 0.$$
Let $(1,N_2,N_3) \in O_N$ and $(1,S_2,S_3) \in O_S$, so we have
\begin{equation} \label{ConicRelation}
N_2^2=-N^{-1}-c N_3^2 \text{ and } S_2^2=-S^{-1}-c S_3^2. 
\end{equation}
The corresponding tangents are given by
$$ t_{O_N}(1,N_2,N_3)=(1,NN_2,cNN_3) \text{ and } t_{O_S}(1,S_2,S_3)=(1,SS_2,cSS_3). $$
Assume   these  tangents   are  the   same,  i.e.   $(1,NN_2,cNN_3)  =
(1,SS_2,cSS_3)$.  Since they  both have  the same  entry at  the first
coordinate, we necessarily have
\begin{equation} \label{TangentCondition}
NN_2=SS_2 \text{ and } cNN_3=cSS_3.
\end{equation}
Combining  (\ref{ConicRelation})  and  (\ref{TangentCondition})  gives
indeed $N=S$ which completes the proof.
\end{proof}
We end  this section by showing  that the parameter $c$  can indeed be
chosen  arbitrarily  among  all  squares or  nonsquares,  without  any
changes of incidence relations.
\begin{lemma}  \label{NotDependC} Let  $c_1$ and  $c_2$ be  squares in
  case  $p\equiv   3(4)$  and  nonsquares  in   case  $p\equiv  1(4)$,
  respectively.  Then the two partitions of $PG(2,p)$ given by
$$
 x^2 = 0,\quad  y^2 + c_iz^2 = 0,\quad  x^2 + ky^2 + kc_iz^2 = 0, \ k = 1, \ldots , p - 1
$$
for $i=1,2$ can be mapped to each  other by a collinear transformation of
$PG(2,p)$.
\end{lemma}
\begin{proof}   In both case,
  $c_1 c_2$ is a  square, so we can consider the collinear map $\phi_S$  given by the
   matrix
$$ S = \begin{pmatrix} 1 & 0 & 0 \\ 0 & 1 & 0 \\ 0 & 0 & \sqrt{c_2{c_1}^{-1}}\end{pmatrix}. $$ 
If $O$  is the matrix corresponding  to one of the  equations defining
the partition with parameter $c_1$,  then $S^T O S$ is the
matrix of the corresponding equation with parameter $c_2$.
\end{proof}
\subsection{Poncelet's Theorem for conics $O_k$}

The main  goal in  this section  is to prove  a version  of Poncelet's
Porism, interpreted in  $PG(2,p)$. Recall that we  are only interested
in pairs of conics of  the form (\ref{ConicEquation}) described in the
previous section.
\begin{definition}  Consider  a  pair of  conics  $(O_\alpha,O_\beta)$
  given by (\ref{ConicEquation}). An \emph{$n$-sided Poncelet Polygon}
  is a  polygon with $n$  sides and  vertices on $O_\beta$,  such that
  the sides are tangents of $O_\alpha$.
\end{definition}
Since the conics  $O_k$ are all disjoint and have  no common tangents,
as  seen in  Lemma \ref{disjointtangents},  we are  in the  particular
situation  that  if we  can  find  one line,  which  is  a tangent  to
$O_\alpha$  and a  secant of  $O_\beta$, this  leads necessarily  to a
Poncelet Polygon.  The  version of Poncelet's Theorem we  are going to
prove here reads as follows:
\begin{theorem} \label{Poncelet} Let  $(O_\alpha,O_\beta)$ be any pair
  of conics in $PG(2,p)$ with equations of the form
$$ O_k: x^2 + k y^2 + c k z^2,\ k \in \left\lbrace \alpha, \beta \right\rbrace. $$
If an  $n$-sided Poncelet Polygon  can be constructed starting  with a
point  $P \in  O_\beta$, then  an  $n$-sided Poncelet  Polygon can  be
constructed starting with any other point $Q \in O_\beta$ as well.
\end{theorem}

\begin{proof} We start with the  result in planes $PG(2,p)$, $p \equiv
  3(4)$. By Lemma \ref{NotDependC}, it  suffices to consider conics of
  the form
\begin{equation} \label{CircleEquation}
 O_k: x^2+k y^2+k z^2 =0.
\end{equation}
since $c=1$ is always a square. The idea is to find a collineation, which does not change the equation of the conics but maps points of the conic to each other.
Since the conics in this case can be interpreted as concentric  circles, we apply collineations which rotate the
$n$-sided  Poncelet  Polygons suitably, i.e. we look at members of the dihedral group of $p+1$ elements or a subgroup of it. Let $P_1,P_2,...,P_n$ be points on $O_\beta$ which form an $n$-sided Poncelet Polygon with some other conic $O_\alpha$. Let $Q=(1,q_2,q_3)$ be any other point on $O_\beta$. Denote $P_1=P=(1,p_2,p_3)$. The goal is to find a collineation $\tau=\tau_{(P,Q)}$ which maps $P$ to $Q$ and does not change the conic equation. For this, look at:
$$ \tau = \begin{pmatrix} 1 & 0 & 0 \\ 0 & a & b  \\ 0 & b & -a  \end{pmatrix} $$
We want $\tau(P)=Q$ which leads to the two conditions:
\begin{equation} \label{Crit1}
ap_2+bp_3=q_2
\end{equation}
\begin{equation}\label{Crit2}
bp_2-ap_3=q_3
\end{equation}
We know that $P$ and $Q$ are on $O_\beta$, which means:
\begin{align} \label{ovalEqu}
p_2^2+p_3^2=-k^{-1} \text{ and } q_2^2+q_3^2=-k^{-1}
\end{align}
In the case $p_2 \neq 0$ and $p_3\neq 0$, we immediately deduce by combining (\ref{Crit1}) and (\ref{Crit2}):
$$a = -kp_2q_2+kq_3p_3 \text{ and } b=-kq_2p_3-kq_3p_2$$
Note that $a$ and $b$ are elements in $GF(p)$ with the property $a^2+b^2 \equiv 1(p)$, since using (\ref{ovalEqu}) gives:
$$ a^2 + b^2 = k^2(q_2^2+q_3^2)(p_2^2+p_3^2)=1 $$
Because of this, $\tau$ does not change the equation of the conic, i.e. $ \tau(O_k) = O_k$, as  can be easily checked.
Hence, applying $\tau_{(P,Q)}$ to all points $P_i$ for $i=1,...,n$ yields a new Poncelet Polygon using the point $Q$. Similarly, the cases for $P=(1,0,p_3)$ and $P=(1,p_2,0)$ can be discussed, which completes the proof for planes  $PG(2,p)$,  $p  \equiv 3(4)$.
 
It  remains  to  prove  the  result in  planes  $PG(2,p)$,  $p  \equiv
1(4)$. Remember that in these planes,  all nonsquares can be taken for
the parameter $c$.  But $c=1$ is a  square, and we have to  set up the
transformation differently. Let  $c$ be an arbitrary but fixed
nonsquare in $GF(p)$. Define
$$ GF(p)(\sqrt{c}):=\left\{a+b\sqrt{c}\ | \ a,b \in GF(p) \right\} $$
and consider the usual addition and multiplication, defined by:
$$ (x+y\sqrt{c}) + (z+w\sqrt{c}) := (x+z) + (y+w)\sqrt{c} $$
$$ (x+y\sqrt{c})(z+w\sqrt{c}) := (xz+ywc) + (yz+xw)\sqrt{c} $$
It  is  well-known that  $  GF(p)(\sqrt{c})$  is  indeed a  field  and
isomorphic  to  $GF(p^2)$.   Let  us  go  back  to  the  conics  $O_k:
k^{-1}x^2+y^2+c z^2=0$ in $PG(2,p)$ we  started with. The main idea is
to   embed   these   conics   in   $PG(2,p^2)$.   First   note,   that
$\varepsilon:PG(2,p)\to   PG(2,p^2),   P\mapsto   P$  is   a   natural
embedding.  Then we  choose the  collinear transformation  $\phi_s$ in
$PG(2,p^2)$ given by the matrix $S$ in $GF(p)(\sqrt{c})$:
$$ S = \begin{pmatrix} 1 & 0 & 0 \\ 0 & 1 & 0  \\ 0 & 0 & \sqrt{c}  \end{pmatrix} $$
All conics $O_k$ in $PG(2,p)$  are mapped by $\phi_S\circ \varepsilon$
to conics  $S(O_k)$ in $PG(2,p^2)$.  The equation of $S(O_k)$  is then
given by
$$ S(O_k): k^{-1} x^2+y^2+z^2 = 0,$$
since
$$ (S^{-1})^T O_k S^{-1} = 
\begin{pmatrix} k^{-1} & 0 & 0 \\ 0 & 1 & 0  \\ 0 & 0 & 1  \end{pmatrix}.
$$
For the  tangents of  $O_k$, we can  proceed similarly  by considering
again the equation of $S(O_k)$  in $PG(2,p^2)$. It is easy to check, that a tangent of $O_k$ is mapped to
a tangent of $S(O_k)$ by $\phi_S\circ \varepsilon$. Now, on the conics $S(O_k)$ in $PG(2,p^2)$, we can
operate similarly as before. Let $P_1,P_2,...,P_n$ be points on $O_\beta$ which form an $n$-sided Poncelet Polygon with some other conic $O_\alpha$. Let $Q=(1,q_2,q_3)$ be any other point on $O_\beta$ and denote $P_1=P=(1,p_2,p_3)$. We look at $P$ and $Q$ in $PG(2,p^2)$, namely at $SP=(1,p_2,\sqrt{c}p_3)$ and $SQ=(1,q_2,\sqrt{c}q_3)$. The goal is to find a collineation $\tau=\tau_{(SP,SQ)}$ which maps $SP$ to $SQ$ and does not change the conic equation $S(O_k)$. Now, we look at:
$$ \tau = \begin{pmatrix} 1 & 0 & 0 \\ 0 & a & \sqrt{c}b  \\ 0 & \sqrt{c}b & -a  \end{pmatrix} $$ 
The condition $\tau(SP)=SQ$ leads to:
$$ap_2+cbp_3=q_2 \text{ and } bp_2-ap_3=q_3$$
Assume again $p_2 \neq 0$ and $p_3 \neq 0$. The other two cases (i.e. $p_2=0$ or $p_3=0$) can be carried out similarly. 
\begin{equation}
a = -kp_2q_2+kcq_3p_3
\end{equation}
\begin{equation}
b=-kq_2p_3-kq_3p_2
\end{equation}
It can be checked immediately, by using $P,Q \in O_\beta$, that:
\begin{align} \label{help1}
a^2+cb^2=1
\end{align}
Hence, $\tau$ does not change the equation of $S(O_k)$ and therefore maps points on $S(O_k)$ to points on $S(O_k)$. 
We apply $\tau$ to all points on the embedded Poncelet $n$-gon given by $SP_1,...,SP_n$, which gives a new Poncelet $n$-gon for $S(O_\beta)$ and $S(O_\alpha)$. It remains to show that mapping these transformed points back to $PG(2,p)$ gives points on the original conic $O_\beta$, i.e. we have to show that for any $P_i=(1,y_i,z_i) \in O_\beta$:
$$ S^{-1}(\tau(SP_i)) = \begin{pmatrix} 1 \\ ay_i+cbz_i  \\ by_i-az_i \end{pmatrix} \in O_\beta$$
Look at the corresponding equation:
$$ 1+ k(ay_i+cbz_i)^2+ck(by_i-az_i)^2 = 1+ky_i^2(a^2+cb^2)+kcz_i^2(a^2+cb^2)$$
Because of (\ref{help1}) and $P_i \in O_\beta$ this expression indeed equals zero. Hence, we end up with a Poncelet $n$-gon for $O_\beta$ and $O_\alpha$ starting with the point $Q$.
\end{proof}

\subsection{Relations for pairs of conics}
In this  section, we consider the  disposition of pairs of  conics with
regard to the existence of a Poncelet Polygon.
\begin{definition}  Let  $O$ and  $O'$  be  two conics  in
  $PG(2,p)$. We  say that \emph{$O$ lies  inside $O'$}, if
  $O'$ consists of external  points of $O$ only. Notation:
  $O \diamond O'$.
\end{definition}
Note  that this  relation is  not  symmetric, since  there are  conics
$O$ consisting of external points of $O'$ but the converse
is  not  true.  Moreover,  we  can have  the  situation  that  neither
$O$ lies inside $O'$ nor $O'$ lies inside $O$.

In  the following  we continue  to  consider conics  described in  the
previous section, i.e., conics given  by the equation $O_k: x^2+k y^2+
c k z^2 = 0$.
\begin{theorem} \label{diamond}  Let $O_\alpha$  and $O_\beta$  be
  conics in  $PG(2,p)$ of  the given  form, $p$ an  odd prime.  If one
  point  $P \in  O_\beta$ is  an  external point  of $O_\alpha$,  then
  $O_\alpha \diamond  O_\beta$. Moreover,  we have $O_\alpha \diamond
  O_\beta$ if and only if:
$$ (-\beta) (\beta-\alpha) = \begin{cases} \text{nonsquare in $GF(p)$}, & p \equiv 1 (4) \\ \text{square in $GF(p)$}, & p \equiv 3 (4)\end{cases}$$
\end{theorem}
\begin{proof} Let $O_\alpha$ and $O_\beta$ be given by:
$$ O_\alpha: x^2+\alpha y^2+ c \alpha z^2 = 0 \text{ and } O_\beta: x^2+\beta y^2+ c \beta z^2 = 0$$
Remember that all points of $PG(2,p)$ with a zero $x$-coordinate lie on the line $g: x^2=0$, hence due to the partition not on any conic. A point $P$ of $O_\beta$ can therefore be considered as $P=(1,P_2,P_3)$. Using the conic equation, we have $P_2^2 =
-\beta^{-1}-cP_3^2$.  By Lemma  \ref{AP},  the  conic $O_\alpha$  lies
inside $O_\beta$ if for all such points $P$, $O_\alpha P$
is a secant of $O_\alpha$. So,  the property $O_\alpha \diamond
O_\beta$ can be equivalently expressed as follows:
\begin{align*}
(1, \alpha P_2,c \alpha P_3) \text{ a secant of $O_\alpha$ } &\Leftrightarrow \\
 \exists (x,y,z) \in O_\alpha: \alpha^{-1} x + P_2 y + cP_3 z = 0 &\Leftrightarrow \\
\exists (x,y,z) = (1, \pm \sqrt{-\alpha^{-1}-cz^2} ,z): \alpha^{-1} x + P_2 y + cP_3 z = 0 &\Leftrightarrow\\
 \pm \sqrt{(-\beta^{-1}-cP_3^2)(-\alpha^{-1}-cz^2)} = -\alpha^{-1}-cP_3z &\Leftrightarrow \\
 z^2 - 2 \alpha^{-1} \beta P_3 z + \alpha^{-1}c^{-1} + \alpha^{-1}\beta P_3^2-\alpha^{-2}\beta c^{-1}=0 
\end{align*}
This quadratic equation is solvable for  $z$ iff its discriminant is a
square in $GF(p)$, i.e., iff
\begin{align*}
\alpha^{-2}\beta^2 P_3^2-\alpha^{-1} c^{-1}-\alpha^{-1}\beta P_3^2+\alpha^{-2} \beta c^{-1} \text{ square } &\Leftrightarrow \\
\beta^2 P_3^2-\alpha c^{-1}-\alpha\beta P_3^2+ \beta c^{-1} \text{ square } &\Leftrightarrow \\
(\beta P_3^2+c^{-1})(\beta-\alpha) \text{ square } &\Leftrightarrow \\
 (-\beta P_2^2 c^{-1})(\beta - \alpha) \text{ square } &\Leftrightarrow \\
(-\beta c^{-1})(\beta - \alpha) \text{ square }
\end{align*}
In  planes $PG(2,p)$,  $ p  \equiv 1  (4)  $, $c$  is chosen  to be  a
nonsquare,  hence  we   need  $(-\beta)(\beta  -  \alpha)$   to  be  a
nonsquare. In planes  $PG(2,p)$, $ p \equiv 3 (4)  $, the variable $c$
is  a  square, hence  we  need  $(-\beta)(\beta  -  \alpha)$ to  be  a
square. Since the above expression is independent of the point $P$, it
holds for every point on $O_\beta$ and we are done.
\end{proof}
With  Theorem \ref{diamond},  we can  now construct  chains of  nested
conics, since we have:
\begin{corollary}  \label{odot}  Consider  two conics  $O_\alpha$  and
  $O_\beta$. Then:
$$ O_\alpha \diamond O_\beta \Leftrightarrow O_\beta \diamond O_{\beta^2 \alpha^{-1}} $$
\end{corollary}

When calculating the relation $\diamond$ for every pair in a given plane, it is useful to apply the following result. 
\begin{lemma} \label{timesBeta} Let $(O_k,O_1)$ be a pair of conics in $PG(2,p)$. Then there exists a collinear transformation mapping $(O_k,O_1)$ to $(O_{\beta k},O_\beta)$, for all $\beta \in GF(p)\backslash \left\{0\right\}$. In particular, $O_k \diamond O_1$ implies $O_{\beta k} \diamond O_\beta$.
\end{lemma}
\begin{proof} We have to find a collinear transformation $\phi_S$ with:
$$ \phi_S(O_k) = O_{\beta k} $$
Let us start with the case of $\beta$ being a square in $GF(p)$.
In this case, we are allowed to consider the square root of $\beta$, hence the following regular matrix $S$ works, which can be checked immediately:
$$ S=  \begin{pmatrix} 1 & 0 & 0 \\ 0 & \sqrt{\beta}^{-1} & 0 \\ 0 & 0 & \sqrt{\beta}^{-1}  \end{pmatrix} $$
For $\beta$ not a square, we have to distinguish between $p \equiv 3 (4)$ and $p \equiv 1 (4)$. In the first case, we can take any square for the parameter $c$ in the conic equation, as seen in Lemma \ref{diamond}. Since the relation of the conics does not depend on $c$ by Lemma \ref{NotDependC}, we can take $c=1$. Now, choose any nonzero square $s$ in $GF(p)$ such that $\beta -s$ is a square as well. Note that we can always find such an $s$, because $\beta-s$ runs for $s$ a square through exactly $\frac{p-1}{2}$ values, since there are $\frac{p-1}{2}$ squares in $GF(p)\backslash \left\{0\right\}$. Then $\beta-s \neq \beta$, since we take $s$ not to be zero. Hence, not all of those $\frac{p-1}{2}$ values can be nonsquares. Moreover, $\beta-s\neq 0$ as $s$ is a square and $\beta$ a nonsquare. Hence, there is a nonzero square $\beta-s$. Therefore, the square roots of $s$ and $\beta-s$ are well-defined. Now, look at the following matrix:
$$ S^{-1} = \begin{pmatrix} 1 & 0 & 0 \\ 0 & \sqrt{s} & -\sqrt{\beta-s} \\ 0 & \sqrt{\beta-s} & \sqrt{s}  \end{pmatrix} $$
We have $\det(S^{-1}) = \beta \neq 0$ and moreover, it holds:
$$ \begin{pmatrix} 1 & 0 & 0 \\ 0 & \sqrt{s} & \sqrt{\beta-s} \\ 0 & -\sqrt{\beta-s} & \sqrt{s}  \end{pmatrix} \begin{pmatrix} 1 & 0 & 0 \\ 0 & k & 0 \\ 0 & 0 & k  \end{pmatrix} \begin{pmatrix} 1 & 0 & 0 \\ 0 & \sqrt{s} & -\sqrt{\beta-s} \\ 0 & \sqrt{\beta-s} & \sqrt{s}  \end{pmatrix} = \begin{pmatrix} 1 & 0 & 0 \\ 0 & k \beta & 0 \\ 0 & 0 & k \beta  \end{pmatrix} $$
This gives the valid coordinate transformation we were looking for. In planes $PG(2,p)$, $p \equiv 1 (4)$, we can take any nonsquare for our parameter $c$. By Lemma \ref{NotDependC}, the relation of the conics does not depend on $c$, hence we choose a $c$, such that $\beta-c$ is a nonzero square. To ensure the existence can be proceeded as before. For such a $c$, look at the matrix:
$$ S^{-1} = \begin{pmatrix} 1 & 0 & 0 \\ 0 & \sqrt{\beta-c} & c \\ 0 & 1 & -\sqrt{\beta-c}  \end{pmatrix}  $$
We have $\det(S^{-1}) = -\beta \neq 0$ and it can be checked that this is the collinear transformation we are looking for.
\end{proof}

\begin{example}  We want  to  investigate the  relation $\diamond$  in
  $PG(2,7)$.  Here, $p \equiv 3 (4)$, hence $O_\alpha \diamond O_\beta$
  if and  only if  $(-\beta)(\beta-\alpha)$ is a  square in $GF(7)$, i.e. equals 1,2 or 4. By looking at $\beta=1$ and shifting the result using Lemma \ref{timesBeta}, we obtain the table of relations for the whole plane $PG(2,7)$.
\begin{center}
\begin{tabular}{c| c| c |c| c |c| c}
 & $O_1$ & $O_2$ & $O_3$ & $O_4$ & $O_5$ & $O_6$\\
\hline 
$O_1$  &  &  & $\diamond$ &$\diamond$& $\diamond$ &  \\ 
\hline
$O_2$ & $\diamond$ &  & $\diamond$ &  &  & $\diamond$ \\
\hline
$O_3$ & $\diamond$ & $\diamond$ &  &  & $\diamond$ &  \\
\hline
$O_4$  &  & $\diamond$ &  &  & $\diamond$ & $\diamond$ \\
\hline
$O_5$  & $\diamond$ & &  & $\diamond$ &  & $\diamond$ \\
\hline
$O_6$ &  & $\diamond$ & $\diamond$ & $\diamond$ &  &  \\
\end{tabular}
\end{center}
Using  Corollary \ref{odot},  we detect the following  closed chains of  conics
$O_\alpha  \rightarrow  O_\beta  \rightarrow  O_{\beta^2  \alpha^{-1}}
\rightarrow ...$
\begin{align*}
	O_1 \rightarrow O_3 \rightarrow O_2 \rightarrow O_6 \rightarrow O_4 \rightarrow O_5 \rightarrow O_1 \\
	O_1 \rightarrow O_4 \rightarrow O_2 \rightarrow O_1  \\
	O_3 \rightarrow O_5 \rightarrow O_6 \rightarrow O_3  
\end{align*}
Note that starting with two squares  $\alpha$ and $\beta$ results in a
chain of conics with just squares as indices. Similarly, starting with
two  nonsquares as  indices results  in a  chain of  conics with  only
nonsquares as indices.  This shows a connection of  this property with
cyclotomic subsets, defined by
$$ C_i^{p}(q):= \left\{q i^n (p),1 \leq n \leq p-1 \right\}.  $$
In     $GF(7)$,    we     have    for     example    $C_3^{7}(1)     =
\left\{3,2,6,4,5,1\right\}$,  $C_2^{7}(1) =  \left\{2,4,1\right\}$ and
$C_2^{7}(3) = \left\{6,5,3\right\}$.
\end{example}
Since exactly  half of  all nonzero elements  in $GF(p)$  are squares,
another immediate result is:
\begin{corollary} \label{RelConics} For  every conic  $O_\beta$ in $PG(2,p)$,  there are
  $\frac{p-1}{2}$  conics  $O_\alpha$  such  that  $O_\alpha  \diamond
  O_\beta$.
\end{corollary}
Next,  we  have a  closer  look  at the  relations  of  the points  on
$O_\alpha$ and $O_\beta$.
\begin{lemma}  \label{polarintersection}  Let $P=(1,P_2,P_3)$  be  any
  point on  $O_\beta$ and $O_\alpha  \diamond O_\beta$. Then,  for the
  contact points $A_1=(1,y_1,z_1)$ and $A_2=(1,y_2,z_2)$ on $O_\alpha$
  of the tangents through $P$ we have
$$ z_{1,2} = \alpha^{-1} \beta P_3 \pm P_2 \sqrt{\alpha^{-2} (-c^{-1} \beta)(\beta-\alpha)}$$
and
$$ y_{1,2} =  \begin{cases} P_2^{-1}(-\alpha^{-1}-cP_3z_{1,2}), &\text{if $P_2 \neq 0$}\\ \pm \sqrt{-\alpha^{-1}-cz_{1,2}^2},&\text{if $P_2 = 0$.}\end{cases} $$
\end{lemma}
\begin{proof}  To  see this,  we  just  have  to solve  the  quadratic
  equation derived in Theorem  \ref{diamond}. Since $O_\alpha \diamond
  O_\beta$, we indeed get two solutions.
\end{proof}

\begin{lemma}  \label{PlusOnA}  Let  $P$  and $Q$  be  two  points  on
  $O_\beta$ such that the line connecting  $P$ and $Q$ is a tangent of
  $O_\alpha$ in the point $A$. Then
$$ P + Q = A. $$
\end{lemma}
\begin{proof} Let $P=(1,P_2,P_3)$ and $Q=(1,Q_2,Q_3)$ be two points on
  $O_\beta$, so we have:
$$ 1 + \beta P_2^2 + c \beta P_3^2 = 0 \text{ and } 1 + \beta Q_2^2 + c \beta Q_3^2 = 0$$
There are $p+1$ points on the line through $P$ and $Q$, namely
$$ \overline{PQ} = \left\{ P,Q,P+Q,P+2Q,...,P+(p-1)Q \right\}. $$
Note that  $P+(p-1)Q$ has  a zero $x$-coordinate  and hence  gives the
intersection with the line $g: x^2=0$, which is clearly not a point on
$O_\alpha$. So we know that
\begin{equation} \label{UniqueK}
\exists! \ k \in \left\{1,2,...,p-2\right\}: P+kQ = (1+k,P_2 + kQ_2, P_3 + kQ_3)= A \in O_\alpha
\end{equation}
We need:
\begin{align*}
(k+1,P_2 + kQ_2, P_3 + kQ_3) \in O_\alpha &\Leftrightarrow \\
(k+1)^2 + \alpha (P_2 + kQ_2)^2 + c \alpha (P_3 + k Q_3)^2 = 0 &\Leftrightarrow \\
k^2 + \frac{2+2\alpha(P_2Q_2+cP_3Q_3)}{1-\alpha \beta^{-1}}k + 1 = 0 
\end{align*}
Note   that   $1-\alpha   \beta^{-1}    \neq   0$,   since   otherwise
$\alpha=\beta$. Solving for $k$ yields
$$ k = - \frac{1+\alpha(P_2Q_2+cP_3Q_3)}{1-\alpha \beta^{-1}} \pm \sqrt{\left(\frac{1+\alpha(P_2Q_2+cP_3Q_3)}{1-\alpha \beta^{-1}}\right)^2-1}. $$
Note that  by (\ref{UniqueK}), there can  only be one solution  to our
problem, since we are looking for a tangent of $O_\alpha$. Because of
this, the  radicand has to be zero,  which is if
and     only    if     $\left(\frac{1+\alpha(P_2Q_2+cP_3Q_3)}{1-\alpha
    \beta^{-1}}\right)^2  =  1$.  Hence  $k=1$ or  $k=p-1$.  Since  we
already excluded $p-1$, we obtain $k=1$, which indeed gives $P+Q=A$.
\end{proof} 
\begin{corollary}  Let $P,Q  \in  O_\beta$ such  that $(1,0,0)  \notin
  \overline{PQ}$.      Then     there      exists  an   $\alpha      \in
  \left\{1,2,...,p-1\right\}$,   $\alpha   \neq  \beta$,   such   that
  $\overline{PQ}$ is  a tangent of  $O_\alpha$. The contact point is $P+Q$.
\end{corollary}
\begin{proof} With $P=(1,P_2,P_3), Q=(1,Q_2,Q_3)$, we have $P+Q=(2,P_2+Q_2,P_3+Q_3)$.  As the
  characteristic of $GF(p)$ is odd, $P+Q$  is not in $g$, where $g$ is
  the unique  line through  $(0,1,0)$ and $(0,0,1)$.  Since we  have a
  partition of the plane $PG(2,p)$, $P+Q$ must be a point on a conic
  $O_\alpha$. We  have to  exclude the possibility  of $\overline{PQ}$
  being a  secant of $O_\alpha$. For  this, note that there  are $p+1$
  points on $\overline{PQ}$, where $P,Q \in O_\beta$ and $P+(p-1)Q \in
  g$. Hence, there are $p-2$ points  left, which is an odd number. All
  the other $p-2$ points must lie on  conics and there are at most two
  points  on  the  same  conic.  Since  $p-2$  is  odd  and  by  Lemma
  \ref{disjointtangents}, there is exactly  one conic with $\overline{PQ}$ as
  a tangent. By Lemma \ref{PlusOnA}, we are done.
\end{proof}
In the following results, an  $n$-sided Poncelet Polygon for $O_\alpha
\diamond  O_\beta$  with  vertices $B_i$  on  $O_\beta$  and  contact points $A_i$  on
$O_\alpha$ is denoted by
 $$ B_1 \stackrel{A_1}{\longrightarrow} B_2 \stackrel{A_2}{\longrightarrow} B_3 \stackrel{A_3}{\longrightarrow} ... \stackrel{A_{n-1}}{\longrightarrow} B_n \stackrel{A_n}{\longrightarrow} B_1,$$
 where $  B_i \stackrel{A_i}{\longrightarrow} B_{i+1}$ means  that the
 line connecting $B_i$  and $B_{i+1}$ is the tangent  of $O_\alpha$ in
 the  point $A_i$.  By  Lemma \ref{PlusOnA},  the following  relations
 are immediate:
\begin{equation}\label{sums} 
B_1 + B_2 = A_1,\ B_2 + B_3 = A_2,\ ... ,B_{n-1}+B_n =A_{n-1},B_n+B_1 = A_n 
\end{equation}
Note  that  by  combining   Lemma  \ref{polarintersection}  and  Lemma
\ref{PlusOnA},  we are  now able  to calculate  a Poncelet  Polygon by
starting  in a  point  on $O_\beta$.  Before  we analyze  Poncelet
Polygons for different numbers of  sides, we need some more properties
of the points on $O_k$ and their relations.
\begin{lemma} The conics $O_\alpha$ in $PG(2,p)$, $p\equiv 3(4)$, consist of the $p+1$ points
	$$\left\{ \begin{pmatrix}1\\y_1\\z_1\end{pmatrix},\begin{pmatrix}1\\-y_1\\z_1\end{pmatrix},\begin{pmatrix}1\\y_1\\-z_1\end{pmatrix},\begin{pmatrix}1\\-y_1\\-z_1\end{pmatrix},...,\begin{pmatrix}1\\y_{\frac{p+1}{4}}\\z_{\frac{p+1}{4}}\end{pmatrix},\begin{pmatrix}1\\-y_{\frac{p+1}{4}}\\z_{\frac{p+1}{4}}\end{pmatrix},\begin{pmatrix}1\\y_{\frac{p+1}{4}}\\-z_{\frac{p+1}{4}}\end{pmatrix},\begin{pmatrix}1\\-y_{\frac{p+1}{4}}\\-z_{\frac{p+1}{4}}\end{pmatrix}\right\} $$
	if $\alpha$ is a square, and otherwise
	$$\left\{ \begin{pmatrix}1\\y_1\\z_1\end{pmatrix},\begin{pmatrix}1\\-y_1\\z_1\end{pmatrix},\begin{pmatrix}1\\y_1\\-z_1\end{pmatrix},\begin{pmatrix}1\\-y_1\\-z_1\end{pmatrix},...,\begin{pmatrix}1\\y_{\frac{p-3}{4}}\\-z_{\frac{p-3}{4}}\end{pmatrix},\begin{pmatrix}1\\-y_{\frac{p-3}{4}}\\-z_{\frac{p-3}{4}}\end{pmatrix}, \begin{pmatrix}1\\y\\0\end{pmatrix},\begin{pmatrix}1\\-y\\0\end{pmatrix},\begin{pmatrix}1\\0\\z\end{pmatrix},\begin{pmatrix}1\\0\\-z\end{pmatrix} \right\} $$

\end{lemma}
\begin{proof}  Of  course, for  $y\neq  0,\  z\neq  0$, we  have  that
  $(1,y,z)  \in  O_\alpha$  implies   that  $(1,-y,z)  \in  O_\alpha$,
  $(1,y,-z) \in  O_\alpha$ and  $(1,-y,-z) \in  O_\alpha$. So  we just
  have to check  whether $(1,0,z)$ and $(1,y,0)$ are  on $O_\alpha$ as
  well. We have:
\begin{alignat*}{2}
(1,0,z) \in O_\alpha &\Leftrightarrow 1 + c \alpha z^2 = 0 &&\Leftrightarrow z^2 = -\alpha^{-1} c^{-1} \\
(1,y,0) \in O_\alpha &\Leftrightarrow 1 + \alpha y^2 = 0 &&\Leftrightarrow y^2 = -\alpha^{-1} 
\end{alignat*}
As  $p \equiv 3 (4)$,  $c$ is a  square in $GF(p)$ and $p-1$ is not. Hence
these  points lie  on $O_\alpha$  if  and only  if $\alpha$  is not  a
square.
\end{proof}

\begin{lemma} The conics $O_\alpha$ in $PG(2,p)$, $p\equiv 1(4)$, consist of the $p+1$ points
	$$ \left\{ \begin{pmatrix}1\\y_1\\z_1\end{pmatrix},\begin{pmatrix}1\\-y_1\\z_1\end{pmatrix},\begin{pmatrix}1\\y_1\\-z_1\end{pmatrix},\begin{pmatrix}1\\-y_1\\-z_1\end{pmatrix},...,\begin{pmatrix}1\\y_{\frac{p-1}{4}}\\z_{\frac{p-1}{4}}\end{pmatrix},\begin{pmatrix}1\\-y_{\frac{p-1}{4}}\\z_{\frac{p-1}{4}}\end{pmatrix},\begin{pmatrix}1\\y_{\frac{p-1}{4}}\\-z_{\frac{p-1}{4}}\end{pmatrix},\begin{pmatrix}1\\-y_{\frac{p-1}{4}}\\-z_{\frac{p-1}{4}}\end{pmatrix},\begin{pmatrix}1\\y\\0\end{pmatrix},\begin{pmatrix}1\\-y\\0\end{pmatrix}\right\} $$
if $\alpha$ is a square, and otherwise
$$\left\{ \begin{pmatrix}1\\y_1\\z_1\end{pmatrix},\begin{pmatrix}1\\-y_1\\z_1\end{pmatrix},\begin{pmatrix}1\\y_1\\-z_1\end{pmatrix},\begin{pmatrix}1\\-y_1\\-z_1\end{pmatrix},...,\begin{pmatrix}1\\y_{\frac{p-1}{4}}\\z_{\frac{p-1}{4}}\end{pmatrix},\begin{pmatrix}1\\-y_{\frac{p-1}{4}}\\z_{\frac{p-1}{4}}\end{pmatrix},\begin{pmatrix}1\\y_{\frac{p-1}{4}}\\-z_{\frac{p-1}{4}}\end{pmatrix},\begin{pmatrix}1\\-y_{\frac{p-1}{4}}\\-z_{\frac{p-1}{4}}\end{pmatrix},\begin{pmatrix}1\\0\\z\end{pmatrix},\begin{pmatrix}1\\0\\-z\end{pmatrix}\right\} $$
\end{lemma}
{\em Proof.} Again,  for $y\neq 0,\  z\neq 0$, we have  that $(1,y,z)
  \in O_\alpha$ if and only  if $(1,-y,z) \in O_\alpha$, $(1,y,-z) \in
  O_\alpha$ and  $(1,-y,-z) \in O_\alpha $.  So we just have  to check
  whether $(1,0,z)$ and $(1,y,0)$ are on the conics as well. Note that
  $c$ is not  a square and $p-1$  is a square in $GF(p)$,  $p \equiv 1
  (4)$, so similarly to the result before, we get:
\begin{alignat}{3}
(1,0,z) \in O_\alpha &\Leftrightarrow 1 + c \alpha z^2 = 0 &&\Leftrightarrow z^2 = -\alpha^{-1} c^{-1} &&\Leftrightarrow \alpha \text{ not a square}\notag \\
 (1,y,0) \in O_\alpha &\Leftrightarrow 1 + \alpha y^2 = 0 &&\Leftrightarrow y^2 = -\alpha^{-1} &&\Leftrightarrow \alpha \text{ a square}\tag*{\large$\Box$}
\end{alignat}
\begin{corollary} The sum of all points on the conic $O_\alpha$ is $(1,0,0)$.
\end{corollary}
\begin{proof} This can be seen by checking all possible cases above.
\end{proof}

\begin{lemma} \label{relA} Let $ B_1 \stackrel{A_1}{\longrightarrow} B_2 \stackrel{A_2}{\longrightarrow} B_3 \stackrel{A_3}{\longrightarrow} ... \stackrel{A_{n-1}}{\longrightarrow} B_n \stackrel{A_n}{\longrightarrow} B_1$ be an $n$-sided Poncelet Polygon. Then
$$ B_1 + B_2 + ... + B_n = (1,0,0) = A_1 + A_2 + ... + A_n. $$
Moreover, for $n$ even, we have
$$ A_1 + A_3 + ... + A_{n-1} = (1,0,0) = A_2 + A_4 + ... + A_n. $$
\end{lemma}
\begin{proof}
By adding the equations in~(\ref{sums}) in different ways, we obtain three conditions for $n$ even:
\begin{eqnarray}
	B_1 + B_2 + ... +B_n &=& A_1 + A_3 + ... + A_{n-1} \label{1} \\
	B_1 + B_2 + ... +B_n &=& A_2 + A_4 + ... + A_n    \label{2}\\
	2(B_1 + B_2 + ... +B_n) &=& A_1+A_2+...+A_{n-1}+A_n \label{3}
\end{eqnarray}
Combining~(\ref{1})  and~(\ref{3})  gives $A_2  +  A_4  + ...  +  A_n=
(1,0,0)$,  and combining~(\ref{2})  and~(\ref{3}) gives  $A_1 +  A_3 +
... + A_{n-1} = (1,0,0) $. Hence, $B_1 +  B_2 + ... +B_n = A_1 + A_2 +
A_3 +...+A_n =(1,0,0)$. Note that the point $(1,0,0)$ operates here as
a neutral  element concerning addition  of points, since  $(0,0,0)$ is
not a point in $PG(2,p)$. The case  $n$ odd is similar.
\end{proof}

\begin{lemma}  \label{evenNgon} The lines joining opposite vertices $B_i$ and $B_{n+i}$ of a  $2n$-sided Poncelet  Polygon 
meet in $(1,0,0)$. Moreover $ B_i+B_{n+i}=(1,0,0)$.
\end{lemma}
\begin{proof} Again, we can use the relation $B_i+B_{i+1}=A_i$ for the
  $2n$-sided Poncelet  Polygon given by  points $B_i \in  O_\beta$ and
  $A_i \in O_\alpha$. We have:
\begin{align*}
B_1 + B_{n+1} &= A_{2n}-B_{2n}+A_{n}-B_{n}\\
B_1 + B_{n+1} &=A_1-B_2+A_{n+1}-B_{n+2}
\end{align*}
Adding these equations gives
$$ 2(B_1 + B_{n+1}) = B_1 + B_{n+1} = A_1-B_2+A_{n+1}-B_{n+2} + A_{2n}-B_{2n}+A_{n}-B_{n}.$$
Taking all $B_i$ to  the left  and all  $A_i$ to  the right  side gives:
$$ B_1 + B_2 + B_{n}+B_{n+1}+B_{n+2}+B_{2n} = A_1+A_{n}+A_{n+1}+A_{2n} $$
To apply Lemma \ref{relA}, we add the remaining $B_i$ to obtain $(1,0,0)$ on the left side:
$$ (1,0,0) = (A_1+A_{n}+A_{n+1}+A_{2n}) + (B_3 + ...+B_{n-1}+B_{n+3}+...+B_{2n-1}) $$
Rewriting the $A_i$ in terms of $B_i$ gives
$$ (1,0,0) = (1,0,0) + B_1 + B_{n+1}. $$
Hence we get
$$ (1,0,0) = B_1 + B_{n+1}. $$
Similarly can be  proceeded for all remaining  $B_i+B_{n+i}$, $1< i
\leq n$.
\end{proof}
Note that  Lemma~\ref{evenNgon} can  bee seen  as a  generalization of
Brianchon's Theorem \cite{LHNH}.
\section{A Poncelet Criterion}
\subsection{Poncelet Coefficients}
Here is a first result concerning the existence of $n$-sided Poncelet Polygons.
\begin{lemma} \label{3coefficient} Let  $O_\alpha \diamond O_\beta$ be
  two  conics  in $PG(2,p)$  which  carry  a Poncelet  Triangle.  Then
  $4\beta = \alpha \text{ in } GF(p)$.
\end{lemma}
\begin{proof}   Let   $    B_1   \stackrel{A_1}{\longrightarrow}   B_2
  \stackrel{A_2}{\longrightarrow}                                  B_3
  \stackrel{A_3}{\longrightarrow}B_1$ be a Poncelet Triangle. By Lemma
  \ref{PlusOnA}, we therefore have
$$ B_1+B_2=2A_1,\ B_2+B_3=2A_2,\ B_3+B_1=2A_3.$$
Moreover,  by Lemma  \ref{relA}, we  have $B_1+B_2+B_3=(1,0,0)$.  This
gives the following relations:
\begin{equation} \label{TriangleRel}
 B_1+2 A_2 = (1,0,0),\ B_2 + 2 A_3 = (1,0,0),\ B_3+ 2 A_1 = (1,0,0)
 \end{equation}
 As all  lines are  given by  linear combinations  of two  points, the
 conditions in (\ref{TriangleRel}) translate to
$$ (1,0,0) \in \overline{B_1A_2},\ (1,0,0) \in \overline{B_2A_3},\ (1,0,0) \in \overline{B_3A_1}. $$
Since there  are no tangents through  the point $(1,0,0)$, as  seen in
Lemma  \ref{UniqueInnerPointP},  these lines  have  to  be secants  of
$O_\alpha$ and  $O_\beta$. With Theorem \ref{QuSecants},  we know that
$\alpha$ and  $\beta$ are either  both squares or both  nonsquares. To
find  the   remaining  intersection  points   of  $\overline{B_1A_2}$,
$\overline{B_2A_3}$  and   $\overline{B_3A_1}$  with   $O_\alpha$  and
$O_\beta$, consider the points  $\tilde{A_i}$ and $\tilde{B_i}$, where
$ \tilde{P} := (x,-y,-z)$ for a point $ P=(x,y,z)$.
Since  $(1,0,0)   \in  \overline{B_i\tilde{B_i}}$  and   $(1,0,0)  \in
\overline{A_i\tilde{A_i}}$,  for  $i=1,2,3$,  these  are  exactly  the
intersection points  we are looking  for. Note that  this construction
yields another Poncelet Triangle: The second triangle is $ \tilde{B_1}
\stackrel{\tilde{A_1}}{\longrightarrow}                    \tilde{B_2}
\stackrel{\tilde{A_2}}{\longrightarrow}                    \tilde{B_3}
\stackrel{\tilde{A_3}}{\longrightarrow}\tilde{B_1}$, as  visualized in
Figure \ref{twoTriangles}.
\begin{figure}[h]
\begin{center}
\definecolor{zzttqq}{rgb}{0.6,0.2,0}
\begin{tikzpicture}[line cap=round,line join=round,>=triangle 45,x=0.6cm,y=0.6cm]
\clip(-5.07,-2.52) rectangle (14.71,6.27);
\fill[color=zzttqq,fill=zzttqq,fill opacity=0.1] (-2.25,4.86) -- (13.68,1.94) -- (8.38,-1.79) -- cycle;
\fill[color=zzttqq,fill=zzttqq,fill opacity=0.1] (10.68,4.88) -- (-0.19,-1.5) -- (12.4,-0.23) -- cycle;
\draw [rotate around={3.6:(8.42,1.05)}] (8.42,1.05) ellipse (2.07cm and 1cm);
\draw [rotate around={178.99:(4.45,1.96)}] (4.45,1.96) ellipse (5.54cm and 2.45cm);
\draw [color=zzttqq] (-2.25,4.86)-- (13.68,1.94);
\draw [color=zzttqq] (13.68,1.94)-- (8.38,-1.79);
\draw [color=zzttqq] (8.38,-1.79)-- (-2.25,4.86);
\draw [color=zzttqq] (10.68,4.88)-- (-0.19,-1.5);
\draw [color=zzttqq] (-0.19,-1.5)-- (12.4,-0.23);
\draw [color=zzttqq] (12.4,-0.23)-- (10.68,4.88);
\draw (10.68,4.88)-- (8.38,-1.79);
\draw (12.4,-0.23)-- (-2.25,4.86);
\draw (13.68,1.94)-- (-0.19,-1.5);
\begin{scriptsize}
\draw[color=black] (8.34,2.34) node {$O_\alpha$};
\fill [color=black] (13.68,1.94) circle (1.5pt);
\draw[color=black] (14.22,1.89) node {$B_3$};
\fill [color=black] (-2.25,4.86) circle (1.5pt);
\draw[color=black] (-2.13,5.28) node {$B_2$};
\fill [color=black] (8.38,-1.79) circle (1.5pt);
\draw[color=black] (8.33,-2.18) node {$B_1$};
\fill [color=black] (9.91,2.63) circle (1.5pt);
\draw[color=black] (10.37,2.8) node {$A_2$};
\fill [color=black] (11.2,0.19) circle (1.5pt);
\draw[color=black] (11.1,0.59) node {$A_3$};
\fill [color=black] (5.62,-0.06) circle (1.5pt);
\draw[color=black] (5.85,0.3) node {$A_1$};
\fill [color=black] (11.83,1.48) circle (1.5pt);
\draw[color=black] (11.2,1.73) node {$\tilde A_1$};
\fill [color=black] (8.8,-0.59) circle (1.5pt);
\draw[color=black] (8.45,-0.3) node {$\tilde A_2$};
\fill [color=black] (5.85,2.04) circle (1.5pt);
\draw[color=black] (5.85,2.42) node {$\tilde A_3$};
\fill [color=black] (10.68,4.88) circle (1.5pt);
\draw[color=black] (11.0,5.17) node {$\tilde B_1$};
\fill [color=black] (12.4,-0.23) circle (1.5pt);
\draw[color=black] (13.42,-0.48) node {$\tilde B_2$};
\draw[color=black] (11.09,-1.44) node {$O_\beta$};
\fill [color=black] (-0.19,-1.5) circle (1.5pt);
\draw[color=black] (-0.31,-1.85) node {$\tilde B_3$};
\fill [color=black] (9.29,0.85) circle (1.5pt);
\draw[color=black] (9.18,1.2) node {$P$};
\end{scriptsize}
\end{tikzpicture}
	\label{twoTriangles}
	\caption{The triangle $B_1,B_2,B_3$ induces another triangle $\tilde{B_1},\tilde{B_2},\tilde{B_2}$ via $P=(1,0,0)$.}
\end{center}
\end{figure}
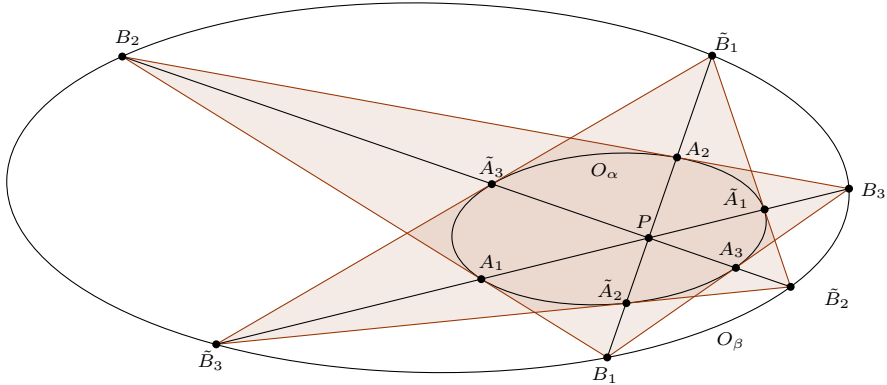
Now, look at $B_1=(1,y_1,z_1)$. The  secant of $O_\beta$ through $B_1$
and $\tilde{B_1}$ is given by
$$ s_1: z_1y-y_1z = 0.$$
In the case $z_1 \neq 0$, we get the relation $y= \frac{y_1}{z_1} z$. Intersecting this line with the conic $O_\alpha$ gives:
$$ 1 + \alpha (\frac{y_1}{z_1} z)^2 + c \alpha z^2 = 0 \Leftrightarrow z^2 = \frac{-z_1^2}{\alpha y_1^2+c \alpha z_1^2} $$
Using $B_1 \in O_\beta$ gives
$$ z^2 = \alpha^{-1} \beta z_1^2.$$
With this, we can calculate  the following two intersection points for
$O_\alpha$:
$$ A_2 = (1,y_1\sqrt{\alpha^{-1}\beta} , z_1\sqrt{\alpha^{-1}\beta} ),\ \tilde{A_2} = (1,-y_1\sqrt{\alpha^{-1}\beta} ,-z_1\sqrt{\alpha^{-1}\beta} ) $$
Using (\ref{TriangleRel}), we obtain the condition
$$ (1+2\sqrt{\alpha^{-1}\beta} )z_1=0$$
Since   we   are   in   the    case   $z_1   \neq   0$,   it   follows
$1+2\sqrt{\alpha^{-1}\beta}=0$, which implies $\alpha  = 4 \beta$.  In
the case  $z_1 = 0$, we  directly deduce $z=0$ for  the secant through
$B_1$  and  $\tilde{B_1}$.  Intersecting  with  $O_\alpha$  gives  the
two points
$$ A_2 = (1, \sqrt{-\alpha^{-1}},0),\ \tilde{A_2} = (1, -\sqrt{-\alpha^{-1}},0) $$
Applying~(\ref{TriangleRel}),  we   get  the  condition  $y_1   \pm  2
\sqrt{-\alpha^{-1}} = 0$  and using $B_1 \in O_\beta$  yields again $4
\beta = \alpha $.
\end{proof}
\begin{remark}
  Recall  that  for $O_\alpha  \diamond  O_\beta$,  we have  to  check
  whether or not $(-\beta)(\beta-\alpha)$ is  a square.  Hence, in the
  case $4 \beta =\alpha$, we have  to check whether or not $3 \beta^2$
  is a square, which is the same  as checking, whether or not $3$ is a
  square. Compared to  results from number theory (see \cite{MR2445243}), we  indeed have the
  following conditions for 3 being a square :
\begin{itemize}
\item For $p \equiv 1(4)$, we have $ 3| (p+1) \Leftrightarrow 3$ nonsquare 
\item For $p \equiv 3(4)$, we have $ 3| (p+1) \Leftrightarrow 3$ square
\end{itemize}
This gives already a necessary condition for the existence of Poncelet
Triangles for  pairs $(O_\alpha,O_\beta)$ in $PG(2,p)$.  By Poncelet's
Theorem  for  such  pairs,  as seen  in  Theorem  \ref{Poncelet},  the
existence of a Poncelet Triangle implies $3|(p+1)$, as there are $p+1$
points on the conic $O_\beta$. This  is exactly the condition given by
number theoretic results as well.
\end{remark}
Using arguments as above, one easily checks the following result.
\begin{lemma}  \label{4gon} Let  $O_\alpha  \diamond  O_\beta$ be  two
  conics in $PG(2,p)$,  such that a $4$-sided Poncelet  Polygon can be
  constructed. Then $ 2\beta = \alpha$ in $GF(p)$.
\end{lemma}

The main  goal is to find  such a relation for  all possible $n$-sided
Poncelet  Polygons. For  this,  we first  investigate, which  Poncelet
$n$-gons occur in a given plane  $PG(2,p)$. Note that this can be done
just by applying Poncelet's Theorem and the Euler divisor sum formula,
since we  are dealing  with a  very special  family, as  the following
results show:

\begin{lemma} \label{howMany} For a given conic $O_\beta$ in $PG(2,p)$ and every $n|(p+1)$, there are exactly $\frac{\phi(n)}{2}$ conics $O_\alpha$, such that $O_\alpha \diamond O_\beta$ carries a Poncelet $n$-gon.
\end{lemma}
\begin{proof}   By   Lemma    \ref{RelConics},   there   are   exactly
  $\frac{p-1}{2}$  conics  $O_\alpha$,  such that  $O_\alpha  \diamond
  O_\beta$. Moreover,  we know that once  $O_\alpha \diamond O_\beta$,
  starting with  any point of  $O_\beta$ leads to a  Poncelet Polygon.
  Because  of  Theorem \ref{Poncelet},  the  length  of this  Poncelet
  Polygon has to divide $p+1$, i.e. the number of points on $O_\beta$.
  Recall now Euler's divisor sum formula for the totient function (see
  \cite{MR2445243}):
$$ \sum_{n|m} \phi(n) = m $$
Applied to the points of the conic, we have: 
$$ \sum_{n|(p+1)} \phi(n) = p+1 \Leftrightarrow \sum_{n|(p+1),\ n\geq 3} \phi(n) = p-1 \Leftrightarrow \sum_{n|(p+1),\ n\geq 3} \frac{\phi(n)}{2} = \frac{p-1}{2} $$
Hence, there have to be exactly $\frac{\phi(n)}{2}$ conics $O_\alpha$ such that $O_\alpha \diamond O_\beta$ carries a Poncelet $n$-gon for every divisor $n$ of $p+1$.
\end{proof}

We are interested in a criterion, which ensures the existence of an $n$-sided Poncelet Polygon for two related conics $O_\alpha \diamond O_\beta$, such as $4\beta=\alpha$ for triangles. 
The next result reduces  this problem  of possibly finding  such a relation for  all $n$-sided Poncelet  Polygons to those with $n$ odd.

\begin{lemma}  \label{DoublingFormula} Let  $(O_{\beta k},O_\beta)$  be  a pair of conics in $PG(2,p)$, which carries an $n$-sided Poncelet Polygon for $k$  a
  square in $GF(p)$. Then $(O_{\beta \tilde{k}},O_\beta)$ carries a $2n$-sided Poncelet Polygon, where
$$ \tilde{k} = \frac{2}{1-\frac{1}{\sqrt{k}}} $$
where only those roots are taken such that $\tilde{k} \neq k$.
\end{lemma}
\begin{proof}  Let  $O_\alpha  \diamond O_\beta$  be some pair of conics which carries  a  $2n$-sided
  Poncelet  Polygon. To  calculate the  relation between  $\alpha$ and
  $\beta$,  we use  that $B_i  + B_{i+1}  = A_i$  for two  consecutive
  vertices of the polygon, as seen in Lemma \ref{PlusOnA}. Hence:
$$ B_1 + B_2 \in O_\alpha \text{ i.e. } (2,y_1+y_2,z_1+z_2) \in O_\alpha$$
This gives the following equation to solve:
$$ 4 + \alpha (y_1 + y_2)^2 + c \alpha (z_1+z_2)^2 = 0 \Rightarrow \alpha = \frac{-4}{(y_1+y_2)^2 + c (z_1+z_2)^2} $$
Since $B_1 \in O_\beta$ and $B_2 \in O_\beta$, we know that $y_i^2 + c z_i^2 = -\beta^{-1}$ for $i=1,2$ and we obtain:
$$ \alpha = \frac{2\beta}{1-\beta(y_1y_2+cz_1z_2)} $$
The claim is $\alpha = \beta \tilde{k}$, hence we have to show:
\begin{align*}
 \frac{2\beta}{1-\beta(y_1y_2+cz_1z_2)} &= \frac{2\beta}{1-\frac{1}{\sqrt{k}}}   \Leftrightarrow \\
1-\beta(y_1y_2+cz_1z_2) &= 1-\frac{1}{\sqrt{k}} \Leftrightarrow \\
\frac{1}{\sqrt{k}} + \beta (-y_1)y_2 + c \beta (-z_1) z_2 &= 0
\end{align*}
The expression above can be interpreted as an inner product. Hence, we
may reformulate the condition as an incidence relation:
\begin{align*}
(\frac{1}{\sqrt{k}},-y_1,-z_1)\cdot(1,\beta y_2,c \beta z_2) = 0 \Leftrightarrow 
(\frac{1}{\sqrt{k}},-y_1,-z_1) \in t_\beta(B_2) 
\end{align*}
This can  be done for all  pairs of points $B_i,B_{i+1}  \in O_\beta$.
We  get  the  following  conditions 
\begin{eqnarray*}
&&(\frac{1}{\sqrt{k}},-y_{2\ell-1},-z_{2\ell-1}),(\frac{1}{\sqrt{k}},-y_{2\ell+1},-z_{2\ell+1}) \in t_\beta(B_{2\ell}) \text{ for }\ell=1,\ldots,n,
\end{eqnarray*}
where indices are taken cyclically.  Exactly $n$ tangents of the conic
$O_\beta$ are involved. The conditions above are equivalent to showing
that the $n$ intersection points are on some conic $O_\gamma$ and form
an $n$-sided Poncelet Polygon with $O_\beta$. Observe that, by Lemma~\ref{evenNgon}, $B_{i+n}=\tilde B_i$,
and hence $(\frac{1}{\sqrt{k}},-y_{i+n},-z_{i+n}) = (\frac{1}{\sqrt{k}},y_{i},z_{i})$. Therefore, we have to verify that
$$ (\frac{1}{\sqrt{k}},\pm y_i,\pm z_i) \in O_\gamma,\ i = 1,...,n, \ \beta =  \gamma k$$
We directly obtain the following equation for $\gamma$:
$$ O_\gamma: \frac{x^2}{k} + \gamma y^2 + c \gamma z^2 = 0$$
Since all  the points  $(1,y_i,z_i)$ lie on  $O_\beta$, we  indeed get
$\beta = \gamma k$. By Lemma \ref{timesBeta} , since $(O_{\beta k},O_\beta)$ carries an $n$-sided Poncelet Polygon, so does $(O_{\gamma k},O_\gamma)$, which is what we wanted to show.
\end{proof}
\begin{corollary} \label{Existence2n} Let  $O_\alpha$ and $O_\beta$ be
  conics in  $PG(2,p)$ such  that a  $2n$-sided Poncelet  Polygon
  exists. Then  there exists  another conic
  $O_\gamma$ such that the pair $(O_\gamma,O_\beta)$ carries an $n$-sided Poncelet Polygon.
\end{corollary}
\begin{proof} Let $O_\alpha \diamond O_\beta$ such that a $2n$-sided Poncelet Polygon can be found and $\alpha = h \beta$. This means that 
$$ (-\beta) (\beta - \alpha) = (-\beta) (\beta - h \beta) = \beta^2(h - 1) $$
is a square in planes $PG(2,p)$, $p \equiv 3(4)$ and a nonsquare in planes $PG(2,p)$, $p \equiv 1(4)$. Since $\beta^2$ is of course a square, we know whether or not $(h-1)$ is a square. To show the statement above, we only have to show that for $\gamma = k \beta$:
$$ h-1 \text{ is a square } \Leftrightarrow k-1 \text{ is a square } $$
To see this, we use our formula for $2n$-sided Poncelet Polygons seen in Lemma \ref{DoublingFormula}:
$$ h - 1  = \frac{2}{1-\frac{1}{\sqrt{k}}} -1 = \frac{\sqrt{k}+1}{\sqrt{k}-1} = \frac{(\sqrt{k}+1)^2}{k-1} $$
This gives us:
$$ (h-1)(k-1) = (\sqrt{k}+1)^2 $$
So we have $(h-1)$ a square if and only if $(k-1)$ a square. 
\end{proof}
\begin{example}
We have already seen in Lemma \ref{4gon} that if $(O_k,O_1)$ forms a $4$-sided Poncelet Polygon, we immediately have $k=2$. Hence by Lemma \ref{DoublingFormula}, we are able to compute the index $h$ such that $(O_h,O_1)$ carries an $8$-sided Poncelet Polygon:
$$ h = \frac{2}{1-\frac{1}{\sqrt{k}}} = \frac{2}{1\pm\frac{1}{\sqrt{2}}} = 4\pm 2 \sqrt{2} $$
This is only well defined if $2$ is a square. For this, we use the following result from number theory  (see \cite{MR2445243}):
\begin{equation} \label{2square}
\text{2 is a square in GF(p)} \Leftrightarrow p \equiv \pm 1 (8) 
\end{equation}
By Poncelet's Theorem, the existence of an $8$-gon already implies $8|(p+1)$. Hence, the condition $p \equiv - 1 (8)$ is again equivalent to a purely number theoretic result.
\end{example}
The next goal is to deduce such relations for all $n$-sided Poncelet Polygons, $n$ odd. The main idea how to proceed lies already in the following result:
\begin{lemma} \label{nextNgon} Let $O_k \diamond O_1$ carry an $n$-sided Poncelet Polygon for the points $B_1,...,B_n \in O_1$, $n$ odd. Then $O_{\frac{k^2}{(k-2)^2}} \diamond O_1$ carries an $n$-sided Poncelet Polygon as well, for the same points $B_1,...,B_n \in O_1$.
\end{lemma}
\begin{proof} Let $O_k \diamond O_1$ such that an $n$-sided Poncelet Polygon can be found for $n$ odd.
By Lemma \ref{PlusOnA}, we have $B_i + B_{i+1} \in O_k$ for all $i = 1,...,n$ and $B_i=(1,y_i,z_i)$, hence:
$$ (2,y_i+y_{i+1},z_i+z_{i+1}) \in O_k \Rightarrow 4 + k (y_i+y_{i+1})^2 + c k (z_i+z_{i+1})^2 = 0$$
Using $1 + y_i^2+c z_i^2 = 0,\ \forall B_i \in O_1$, gives:
$$ k = \frac{2}{1-(y_i y_{i+1} + c z_i z_{i+1})} $$
This implies:
$$ \frac{k-2}{k} + y_i(-y_{i+1}) + c z_i(-z_{i+1}) = 0 $$
which is equivalent to:
$$ \frac{k}{k-2} + \frac{k^2}{(k-2)^2} y_i(-y_{i+1}) +c \frac{k^2}{(k-2)^2} z_i(-z_{i+1}) = 0 $$
Again, this relation can be reformulated as an incidence relation:
$$ (1,y_i,z_i)\cdot \left(\frac{k}{k-2},-\frac{k^2}{(k-2)^2}y_{i+1},-c\frac{k^2}{(k-2)^2}z_{i+1} \right) = 0 $$
Hence we need
$$ (1,y_i,z_i) \in t_{\frac{k^2}{(k-2)^2}}\left(\frac{k}{k-2},-y_{i+1},-z_{i+1}\right) $$
as well as 
$$ (1,y_{i+1},z_{i+1}) \in t_{\frac{k^2}{(k-2)^2}}\left(\frac{k}{k-2},-y_i,-z_i\right) $$
Summarizing gives the condition gives:
$$ (1,y_{i+1},z_{i+1}),(1,y_{i-1},z_{i-1}) \in t_{\frac{k^2}{(k-2)^2}}\left(\frac{k}{k-2},-y_i,-z_i\right) $$
This can be done for all $i = 1,...,n$ and since $n$ is odd, for $O_{\frac{k^2}{(k-2)^2}} \diamond O_1$, an $n$-sided Poncelet Polygon is given via the same points $B_1,...,B_n$.
\end{proof}

An immediate corollary using Lemma \ref{timesBeta} is the following:
\begin{corollary} \label{inverseRel}  The conics $O_1 \diamond O_{\beta^2} $ carry an $n$-sided Poncelet Polygon if and only if $O_{\frac{1}{\beta^2}} \diamond O_1 $ carries an $n$-sided Poncelet Polygon.
\end{corollary}

\begin{remark} We have seen that for triangles, there is only one conic $O_k$ such that $O_k \diamond O_1$ form a $3$-sided Poncelet Polygon, namely $O_4$. In this case, we should therefore have:
$$ k = \frac{k^2}{(k-2)^2}$$
This is equivalent with
$$ k^2 - 5k+4 = 0 \Leftrightarrow k = 1 \text{ or } 4 $$
Hence, we obtain again the condition $k=4$, which we already computed in Lemma \ref{3coefficient} by using other methods.
\end{remark}
The procedure shown in the proof above can be iterated. To avoid long expressions, we have:
\begin{definition} $t_0:=k$, $t_{i+1} := \frac{t_i^2}{(t_i-2)^2}$
\end{definition}
Recall that for a given Poncelet $n$-gon using the points $B_1,...,B_n$ on $O_1$ and tangents of some $O_\alpha$, there are $\frac{\phi(n)}{2}-1$ more conics $O_\gamma$ such that $(O_\gamma,O_1)$ carries an $n$-sided Poncelet Polygon.

\begin{example} \label{5example} We know that for $O_\alpha \diamond O_1$ a $5$-sided Poncelet Polygon for the same five points $B_1,...,B_5 \in O_\alpha$ can be constructed in two different ways, since $\frac{\phi(5)}{2}=2$. Fix an ordering of the points $B_1,...,B_5$ and start with the following polygon: 
$$ B_1 \stackrel{A_1}{\longrightarrow} B_2 \stackrel{A_2}{\longrightarrow} B_3 \stackrel{A_3}{\longrightarrow}B_4 \stackrel{A_4}{\longrightarrow} B_5 \stackrel{A_5}{\longrightarrow} B_1$$
The other $5$-gon is then given by connecting $B_i$ and $B_{i+2}$. 
$$ B_1 \stackrel{C_1}{\longrightarrow} B_3 \stackrel{C_2}{\longrightarrow} B_5 \stackrel{C_3}{\longrightarrow}B_2 \stackrel{C_4}{\longrightarrow} B_4 \stackrel{C_5}{\longrightarrow} B_1$$
Note that connecting $B_i$ and $B_{i+3}$ gives in fact the same polygon again, since we can read the above polygon by reversing the direction.

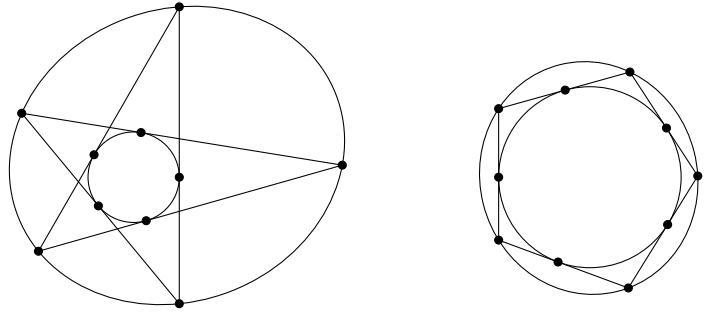
\begin{figure}[h]
\begin{center}
\begin{tikzpicture}[line cap=round,line join=round,>=triangle 45,x=0.6cm,y=0.6cm]
\clip(-3.4600000000000035,-3.3999999999999995) rectangle (13.200000000000014,4.040000000000001);
\draw(0.0,-0.0) circle (0.6cm);
\draw [rotate around={-158.77855620991397:(0.9492090771360864,0.4806415341846275)}] (0.9492090771360864,0.4806415341846275) ellipse (2.2417352840722287cm and 1.935308614177403cm);
\draw (-2.086737653541754,-1.6362264566234326)-- (1.0,3.7655644370746395);
\draw (1.0,3.7655644370746395)-- (1.0,-2.795249781953679);
\draw (1.0,-2.795249781953679)-- (-2.4536025264523214,1.4138282971600855);
\draw (-2.4536025264523214,1.4138282971600855)-- (4.573160707917694,0.26660164665069613);
\draw (4.573160707917694,0.26660164665069613)-- (-2.086737653541754,-1.6362264566234326);
\draw(10.0,0.0) circle (1.2cm);
\draw [rotate around={-76.02376412316792:(9.97457660375866,-0.018022932584529205)}] (9.97457660375866,-0.018022932584529205) ellipse (1.5497133827036382cm and 1.4287558886366318cm);
\draw (8.0,-1.3894582188988451)-- (8.0,1.5150962673243336);
\draw (8.0,1.5150962673243336)-- (10.878051858018814,2.324548352392124);
\draw (10.878051858018814,2.324548352392124)-- (12.364330600626303,0.02757211381691821);
\draw (12.364330600626303,0.02757211381691821)-- (10.843148422973513,-2.448077312559195);
\draw (10.843148422973513,-2.448077312559195)-- (8.0,-1.3894582188988451);
\begin{scriptsize}
\draw [fill=black] (1.0,0.0) circle (1.5pt);
\draw [fill=black] (0.16113189890215207,0.9869328807756821) circle (1.5pt);
\draw [fill=black] (-0.8682431421244593,0.4961389383568336) circle (1.5pt);
\draw [fill=black] (-0.7730733704825902,-0.634316611678023) circle (1.5pt);
\draw [fill=black] (0.2747211278973782,-0.9615239476408232) circle (1.5pt);
\draw [fill=black] (-2.4536025264523214,1.4138282971600855) circle (1.5pt);
\draw [fill=black] (1.0,3.7655644370746395) circle (1.5pt);
\draw [fill=black] (4.573160707917694,0.26660164665069613) circle (1.5pt);
\draw [fill=black] (1.0,-2.795249781953679) circle (1.5pt);
\draw [fill=black] (-2.086737653541754,-1.6362264566234326) circle (1.5pt);
\draw [fill=black] (8.0,0.0) circle (1.5pt);
\draw [fill=black] (9.458508846163433,1.9253018803077981) circle (1.5pt);
\draw [fill=black] (11.679140314304302,1.086502556314549) circle (1.5pt);
\draw [fill=black] (11.704021153838529,-1.0470491427200594) circle (1.5pt);
\draw [fill=black] (9.302125281763523,-1.8742921004065396) circle (1.5pt);
\draw [fill=black] (8.0,1.5150962673243336) circle (1.5pt);
\draw [fill=black] (10.878051858018814,2.324548352392124) circle (1.5pt);
\draw [fill=black] (12.364330600626303,0.02757211381691821) circle (1.5pt);
\draw [fill=black] (10.843148422973513,-2.448077312559195) circle (1.5pt);
\draw [fill=black] (8.0,-1.3894582188988451) circle (1.5pt);
\end{scriptsize}
\end{tikzpicture}
\label{Figure5gons}
\caption{Two different 5-sided Poncelet Polygons can be constructed using the same five points on the outer conic.}
\end{center}
\end{figure}

For $5$-sided Poncelet Polygons, we therefore get the following condition:
$$ t_0 \neq t_1 \text{ and } t_0 = t_2 $$
We have to solve:
$$k = \frac{k^4}{(k^2-2(k-2)^2)^2} $$
which is equivalent to
$$(k-1)(k-4)(16-12k+k^2)=0$$
We obtain the following four solutions:
$$k \in \left\{1,4,6+2\sqrt{5},6-2\sqrt{5} \right\} $$
Since $k=1$ and $k=4$ solves $t_0=t_1$, we find that $k=6\pm 2\sqrt{5}$ implies that if $O_k \diamond O_1$, then $(O_k,O_1)$ carries a $5$-sided Poncelet Polygon. A result by Gauss about quadratic residues (see \cite{MR2445243}) says:
\begin{equation} \label{2square}
\text{5 is a square in GF(p)} \Leftrightarrow p \equiv \pm 1 (5) 
\end{equation}
Hence in all planes $PG(2,p)$, in which 5 actually divides $p+1$, the square root of 5 is well-defined and the indices of the Poncelet $5$-gons given by $6\pm 2\sqrt{5}$ can be computed.
\end{example}

Finally, we can prove the theorem how to find the indices $k$, such that $(O_k,O_1)$ carries an $n$-sided Poncelet Polygon for $n$ odd.
\begin{theorem} Let $n$ be any odd number, $n \geq 3$. Then the indices $k$ such that $(O_k,O_1)$ carries an $n$-sided Poncelet Polygon in a plane $PG(2,p)$ are given by the solutions of:
\begin{equation} \label{ConditionT}
 t_0 = t_{\frac{\phi(n)}{2}},\ t_0 \neq t_i (p)\ \forall i < \frac{\phi(n)}{2} 
 \end{equation}
For a fixed plane $PG(2,p)$, these solutions are called \emph{Poncelet Coefficients} for $n$-sided Poncelet Polygons and denoted by $k^i_n$, $i = 1,...,\frac{\phi(n)}{2}$.
\end{theorem}
\begin{proof} Let $O_k \diamond O_1$ carry an $n$-sided Poncelet Polygon for the points $B_1,...,B_n$, $n$ odd. Let the points be ordered such that for $O_{t_0} \diamond O_1$, we have:
$$B_1 \rightarrow B_2 \rightarrow B_3 \rightarrow ... \rightarrow B_n \rightarrow B_1 $$
We have seen in the proof of Lemma \ref{nextNgon} that the $n$-sided Poncelet Polygon of $O_{t_1} \diamond O_1$ is given by the order: 
$$B_1 \rightarrow B_3 \rightarrow B_5 \rightarrow ... \rightarrow B_n \rightarrow B_2  \rightarrow ... \rightarrow B_{n-1} \rightarrow B_1 $$
Iterating this, we see that the $n$-sided Poncelet Polygon given by $O_{t_i} \diamond O_1$ has the ordering:
$$B_1 \rightarrow B_{1+2^i} \rightarrow B_{1+2*2^i} \rightarrow  B_{1+3*2^i}...  \rightarrow B_1 $$
where the indices are taken cyclically. We already know that there are exactly $\frac{\phi(n)}{2}$ different Poncelet $n$-gons. Since we are only working with $n$ odd, we can apply Fermat's little Theorem (see \cite{MR2445243}) and use
$$ 1+2^{\frac{\phi(n)}{2}}\equiv \pm 1 (n) $$
This shows directly that for $O_{t_{\frac{\phi(n)}{2}}}$ we start the polygon by $B_1 \rightarrow B_2$ or $B_1 \rightarrow B_n$ and hence is equivalent with the very first $n$-gon. To deduce the coefficients $k$ such that $(O_k,O_1)$ carries an $n$-sided Poncelet Polygon, we therefore indeed have to solve (\ref{ConditionT}).
\end{proof}

\begin{remark} Note that for some values of $n$, the iteration needs fewer steps than $\frac{\phi(n)}{2}$, as the order of 2 modulo $n$ can be smaller than $\frac{\phi(n)}{2}$. In these cases, not all indices can be constructed by starting with one Poncelet $n$-gon only. Nevertheless, the condition (\ref{ConditionT}) stays the same but the same coefficients could be derived by computing less, i.e.
$$ t_0 = t_{s_n},\ t_0 \neq t_i (p)\ \forall i < s_n $$
where
$$ s_n := \min\left\{s| 2^s \equiv \pm 1 (n) \right\} $$
The smallest example for $\frac{\phi(n)}{2} \neq s_n$ is $n=17$, where we have $\frac{\phi(17)}{2}=8$ but $2^4 \equiv -1(17)$, i.e. $s_{17}=4$.
\end{remark}
\begin{example} We want to deduce the indices $k$ such that $O_k \diamond O_1$ carries a $9$-sided Poncelet Polygon in $PG(2,53)$. Since $\frac{\phi(9)}{2}=3$, we have to solve:
$$ t_0 = t_3,\ t_1 \neq t_3,\ t_2 \neq t_3 $$
So we need solutions of:
$$ t_0 - t_3 = k - \frac{k^8}{(128+k(-256+k(160+(-32+k)k)))^2} = 0 \mod 53$$
Rewriting this equation, we have to solve:
$$k^8-k(128+k(-256+k(160+(-32+k)k)))^2 = 0 \mod 53 $$
We obtain the following solutions:
$$ k \in \left\{1,4,13,36,40\right\} $$
Since we can exclude the solutions $1$ and $4$, as they also solve $t_2=t_3$, we deduce that:
$$ O_{13} \diamond O_1,\ O_{36} \diamond O_1,\ O_{40} \diamond O_1$$
are the pairs of conics in $PG(2,53)$ such that a $9$-sided Poncelet Polygon can be constructed. 
\end{example}

\subsection{Poncelet Polynomials}
We are now able to give an algorithm to deduce for all pairs $(O_\alpha,O_\beta)$ in $PG(2,p)$, whether an $n$-sided Poncelet Polygon can be constructed for a given $n$. We use the iteration method described before to find polynomials $P_n(k)$ such that the zeros belong to the coefficients $k$ of conics $O_k$, such that if $O_k \diamond O_1$, then $(O_k,O_1)$ carries an $n$-sided Poncelet Polygon. By Lemma \ref{timesBeta}, this gives information about all pairs $(O_\alpha,O_\beta)$.
\begin{definition} A polynomial $P_n$ with integer coefficients is called Poncelet Polynomial for $n$-sided Poncelet Polygons, if the zeros modulo $p$ correspond to the coefficients $k$, such that $O_k \diamond O_1$ carries an $n$-sided Poncelet Polygon in $PG(2,p)$.
\end{definition}

\begin{example} 
We have already seen in Lemma \ref{3coefficient} that $P_3(k)=k-4$ and in Example \ref{5example} that $P_5(k)=16-12k+k^2$.
\end{example}
 By Lemma \ref{howMany}, we know that all these polynomials $P_n$ are of degree $\frac{\phi(n)}{2}$, as the existence of one conic $O_k$, such that $(O_k,O_1)$ carries an $n$-sided Poncelet Polygon in $PG(2,p)$ leads to $\frac{\phi(n)}{2}$ such conics $O_k$. Until now, we only know how to produce Poncelet Polynomials $P_n$ for $n$ odd, but similar to the Poncelet Coefficients $k$, a doubling process can be applied for finding $P_n$ with $n$ even. Note that to find the coefficients for an odd $n$-sided Poncelet Polygon, we look for indices $k$, such that:
$$ P_n(k)=0 \text{ in $GF(p)$} $$
Applying Lemma \ref{DoublingFormula} gives:
$$ P_{2n}(k) = \frac{(k-2)^{\phi(n)} P_n\left( \frac{k^2}{(k-2)^2}\right) }{P_n(k)} $$
\begin{example} We have $P_3(k)=-4+k$ and $\phi(3)=2$. Hence we have to calculate:
$$ (k-2)^2 P_3\left( \frac{k^2}{(k-2)^2}\right) =(k-2)^2 \left( -4+\frac{k^2}{(k-2)^2}\right)  = -4(k-2)^2+k^2 = -(-4+k)(-4+3k) $$
Dividing by $P_3(k)$ gives $ P_6(k)=4-3k $.
\end{example}
For the general case, note that for numbers $n$ and $m$ which have the same value $\phi(n)=\phi(m)$, it has to be checked by hand, which polynomials of degree $\frac{\phi(n)}{2}$ given by the iteration belong to the $n$-gons and which to the $m$-gons. For example, the iteration for $\frac{\phi(n)}{2}=3$ gives the following polynomial:
$$-(-4 + k) (-1 + k) k (-64 + 96 k - 36 k^2 + k^3) (-64 + 80 k -  24 k^2 + k^3)$$
Excluding the factors $(k-4)$ and $(k-1)$ which already occur at the first iteration, we find checking by hand:
$$P_7(k) =-64 + 80 k - 24 k^2 + k^3 \text{ and } P_9(k) =-64 + 96 k - 36 k^2 + k^3$$
With some computational effort, we are now able to create a list of all Poncelet Polynomials $P_n$ up to a chosen value of $n$. Using this list, we can write down an algorithm to find the numbers $k$, such that $(O_k,O_1)$ form an $n$-sided Poncelet Polygon, for all $O_k \diamond O_1$ in a plane $PG(2,p)$. Hence, by recollecting the results described in this discussion, we obtain the algorithm we were looking for, which is described in the following corollary:
\begin{corollary} The following four steps give a complete description of $n$-sided Poncelet Polygons for conic pairs $(O_\alpha,O_\beta)$ in $PG(2,p)$.
\begin{enumerate}
	\item Deduce all $n\geq 3$ with $n|(p+1)$. For every such $n$, calculate $\frac{\phi(n)}{2}$, which gives the number of indices $k$, such that an $n$-sided Poncelet Polygon can be constructed for $(O_k,O_1)$.
	\item For all values $n$ obtained in Step 1, look up the Poncelet Polynomial $P_n$.
	\item For every $P_n$ deduced in Step 2, solve $P_n(k)=0$ modulo $p$. This gives the corresponding Poncelet Coefficients $k$, such that an $n$-sided Poncelet Polygon can be constructed for $(O_k,O_1)$.
	\item By using the coordinate transformation described in Lemma \ref{timesBeta}, transform the information obtained in Step 3 to all pairs $O_\alpha \diamond O_\beta$.
\end{enumerate}
\end{corollary}
 
\begin{example} We want to deduce all relations of conic pairs $O_\alpha \diamond O_\beta$ in the plane $PG(2,11)$. By following the above algorithm, we have:
\begin{itemize}
	\item Step 1: The values $n$, such that an $n$-sided Poncelet Polygon can be constructed, are given by $n=3,4,6,12$. Moreover:
	\begin{center}
\begin{tabular}{c|| c| c |c|c} 
n & 3 & 4 & 6 & 12  \\ 
\hline
$\frac{\phi(n)}{2}$ & 1 & 1 & 1 & 2   
\end{tabular}
\end{center}
	\item Step 2: We have the following Poncelet Polynomials:
\begin{align*}
	P_3(k)&=-4+k\\
	P_4(k)&=-2+k\\
	P_6(k)&=-4+3k\\
  P_{12}(k)&=-16 + 16 k - k^2
\end{align*}
	\item Step 3: The zeros of the Poncelet Polynomials in $GF(11)$ are given by:
	\begin{align*}
	 P_3(k)=0 &\Leftrightarrow k=4\\
	P_4(k)=0 &\Leftrightarrow k=2\\
	P_6(k)=0 &\Leftrightarrow k=5\\
	P_{12}(k)=0 &\Leftrightarrow k=6,10 
	\end{align*}
	\item Step 4: By suitable collinear transformations, we obtain the complete relation table:
	\begin{center}
\begin{tabular}{c|| c| c |c|c|c|c|c|c|c|c} 
 & $O_1$ & $O_2$ & $O_3$ & $O_4$ & $O_5$ & $O_6$ & $O_7$ & $O_8$ & $O_9$ & $O_{10}$  \\ 
\hline
\hline
$O_1$ &  & 12 & 3 & & & 4 & & & 6 & 12\\
\hline
$O_2$ & 4 & & & 12 & & 3 & 6 & & 12 & \\
\hline
$O_3$ & & & & & 6 & 12 & 4 & 12 & 3 & \\
\hline
$O_4$ & 3 & 4 & 6 & & & & 12 & 12 & & \\
\hline
$O_5$ & 6 & & & 3 & & 12 & & 4 & & 12 \\
\hline
$O_6$ & 12 & & 4 & & 12 & & 3 & & & 6\\
\hline
$O_7$ & & & 12 & 12 & & & & 6 & 4 & 3 \\
\hline
$O_8$ & & 3 & 12 & 4 & 12 & 6 & & & &  \\   
\hline
$O_9$ & & 12 & & 6 & 3 & & 12 & & & 4\\
\hline 
$O_{10}$ & 12 & 6 & & & 4 & & & 3 & 12 & 
\end{tabular}
\end{center}
\end{itemize}
\end{example}
\section{Comparison to other methods}

\subsection{Comparison to the Euclidean Plane}
Recall that any point on $O_k: x^2+ky^2+ckz^2=0$ has a nonzero $x$-coordinate. Because of this, we can project these conics on the affine plane by setting $x=1$. Moreover, we can look at real solutions of the equations. In the proof of Poncelet's Theorem for this family of conics, we have seen that there is an affine transformation which maps the whole family into a family of concentric circles. Let us therefore consider pairs of circles in the Euclidean plane, i.e.
\begin{align*}
E_1&: x^2+y^2=1\\
E_r&: x^2+y^2=r^2, r>1
\end{align*}
We are now trying to find a suitable radius $r$ for $E_r$, such that a regular $n$-sided polygon which is inscribed in $E_r$ and circumscribed about $E_1$ can be constructed. It is elementary, that one solution to this problem, namely the circumcircle radius $r$ of a simple, regular $n$-sided polygon is given by:
$$ r= \frac{1}{\cos(\frac{\pi}{n})}$$
In terms of Poncelet Coefficients as defined for the finite case, this gives:
$$k_n = \frac{1}{\cos^2(\frac{\pi}{n})}$$
\begin{example} The radius $r$ for a simple, regular $5$-gon is therefore given by $ r= \frac{1}{\cos(\frac{\pi}{5})}=-1+\sqrt{5}$. Note that $(-1+\sqrt{5})^2=6-2\sqrt{2}$, which is exactly one of the zeros of the Poncelet Polygon for $5$-gons we obtained over finite fields (see Example \ref{5example}). The second radius $\tilde{r}$, which corresponds to the complex $5$-gon circumscribed about $E_1$, can be calculated as well, namely by $\tilde{r}=\frac{1}{\cos(\frac{2\pi}{5})}$, which leads to $\tilde{r}=1+\sqrt{5}$. Hence we obtain $\tilde{r}^2=6+2\sqrt{5}$, which belongs to the second coefficient for $5$-gons obtained in the finite case as well.
\end{example}
Now we turn our attention to the formula deduced for the coefficients $\tilde{k}$ for $2n$-sided Poncelet Polygons in Lemma \ref{DoublingFormula}. For this, note that:
$$ \cos^2\left(\frac{\phi}{2}\right)=\frac{1+\cos(\phi)}{2}$$
Hence we get:
$$ \tilde{k}=\frac{1}{\cos^2(\frac{\pi}{2n})}=\frac{2}{1+\cos(\frac{\pi}{n})}=\frac{2}{1+\frac{1}{\sqrt{k}}} $$
which is exactly the formula derived for the finite case.

Since there does not exist a radical expression for $\cos(\frac{\pi}{n})$ for  all integers $n$, it is convenient to look again at polynomials with roots $\frac{1}{\cos^2(\frac{k\pi}{n})}$. These are closely connected to the $n$-th cyclotomic polynomials $\Phi_n(x)$. Recall, that those polynomials can be written as
$$ \Phi_n(x)=\prod_{1\leq k \leq n,(k,n)=1}(x-e^{\frac{2\pi \i k}{n}}) $$
It is immediate that the degree of $\Phi_n$ is $\phi(n)$, the Euler totient function. The zeros of $\Phi_n(x)$ are given by $e^{\frac{2\pi \i k}{n}}$ for $(k,n)=1$. For a zero $x$ of $\Phi_n$, also $\overline{x}=\frac1x$ is a zero. Define:
$$ q_n(x+\frac{1}{x}):=\Phi_n(x)x^{-\frac{\phi(n)}{2}} $$
The zeros of $q_n$ are then given by $2 \Re(e^{\frac{2\pi \i k}{n}})=2 \cos(\frac{2\pi  k}{n})$. Next, define
$$ r_n(x):=q_n(2x)$$
which has zeros $\cos(\frac{2\pi  k}{n})$. In the next step, we consider
$$ s_n(x):=r_n(2x-1)$$
which has zeros $\frac{1+\cos(\frac{2\pi k}{n})}{2} = \cos^2(\frac{\pi k}{n})$ for $k=1,...,n-1$. Finally, consider
$$P_n(x)=x^{\frac{\phi(n)}{2}} s_n\left(\frac{1}{x}\right)$$
with zeros $\frac{1}{\cos^2(\frac{\pi k}{n})}$, which is exactly the polynomial we wanted. Summarizing, we have
$$ P_n(x)=x^{\frac{\phi(n)}{2}}\Phi_n(z)z^{-\frac{\phi(n)}{2}},\ z=\frac{2-2\sqrt{1-x}}{x}-1 $$
with zeros $\frac{1}{\cos^2(\frac{\pi k}{n})}$ for $(k,n)=1$.
\begin{example} For $n=5$, the cyclotomic polynomial is given by $\Phi_5(x)=x^4+x^3+x^2+x+1$, which leads to
$$ P_5(x)=16-12x+x^2 $$
which indeed has the zeros $6+2\sqrt{5}$ and $6-2\sqrt{5}$. Note that this is the Poncelet Polynomial for $5$-gons derived in the finite case.
\end{example}
\subsection{Comparison to Cayley's Criterion}
The criterion deduced by Cayley in 1853 (see \cite{DR2011}) reads as follows:
\begin{theorem}
Let $C$ and $D$ be the matrices corresponding to two conics generally situated in the projective plane. Consider the expansion
$$ \sqrt{\det(tC+D)}=A_0 + A_1 t+A_2 t^2 + A_3 t^3+...$$
Then an $n$-sided Poncelet Polygon with vertices on $C$ exists if and only if for $n=2m+1$, we have
$$\det \begin{pmatrix} A_2 & ... & A_{m+1} \\ \vdots & ... & \vdots \\ A_{m+1} & ... & A_{2m} \end{pmatrix} = 0 $$
and for $n=2m$, we have
$$\det \begin{pmatrix} A_3 & ... & A_{m+1} \\ \vdots & ... & \vdots \\ A_{m+1} & ... & A_{2m} \end{pmatrix} = 0 $$
\end{theorem}
In the discussion above, we were mainly interested in pairs of conics $(O_k,O_1)$ with equations
$$ O_k: x^2 + k y^2 + ck z^2 =0 $$
$$ O_1: x^2 +  y^2 + c z^2 =0 $$
To apply Cayley's criterion, we therefore have to look at the expansion of the square root of
$$\det \begin{pmatrix} t+1 & 0& 0 \\ 0 & t+k & 0\\ 0 & 0& c(t+k) \end{pmatrix}$$
which is given by
$$\sqrt{c k^2}+ \frac{(k+2) \sqrt{c k^2}}{2 k} t-\frac{ (k-4) \sqrt{c k^2}}{8 k}t^2+\frac{(k-2)  \sqrt{c k^2}}{16 k}t^3-\frac{ (5 k-8) \sqrt{c k^2}}{128 k}t^4+O(t)^5$$
\begin{example} The condition for a $3$-sided Poncelet Polygon is given by vanishing of the coefficient of $t^2$ which is $A_2 = \frac{ (k-4) \sqrt{c k^2}}{8 k}$. This expression is zero if and only if $k-4=0$, which is exactly the condition derived in Lemma \ref{3coefficient} for the finite case.
\end{example}
\begin{example} The condition for $5$-sided Poncelet Polygons is given by $A_2 A_4 - A_3^2=0$, which is the same as $\frac{c ((k-12) k+16)}{1024}=0$. This is equivalent to $k^2-12k+16=0$, so again, we obtain the same condition as for the finite case (compare to Example \ref{5example}).
\end{example}

\section{Conclusion and outlook}

The most interesting result in this discussion is that in finite geometries, the condition whether we can find Poncelet Polygons for a given pair of conics relies on number theoretic results, in particular on the theory of quadratic residues and polynomials over finite fields. It reveals to be the same criterion as in the Euclidean plane, where the conditions are derived by using tools such as angle and length, hence tools which are a priori not available in finite planes. The main aim is therefore to get a better understanding of this connection and try to extend the results for arbitrary pairs of conics in finite projective coordinate planes and to reveal other connections between the theory of quadratic residues and trigonometry. Moreover, we want to understand how the situation changes when considering planes of prime power order rather than planes over prime fields. Also, by having a closer look at chains of the conics described above, more insight into their cyclomic behavior could be obtained, in particular when considering the cyclic representation of the plane.

\addcontentsline{toc}{section}{References}
\bibliography{PonceletSpecialOvals}
\bibliographystyle{unsrt}

\end{document}